\pdfoutput=1
\documentclass[a4paper,11pt]{amsart}

\usepackage[left=1.5in,right=1.5in,top=1in,bottom=1in,includehead,includefoot]{geometry}
\usepackage[utf8]{inputenc}
\usepackage[T1]{fontenc}
\usepackage{lmodern,microtype}
\usepackage[leqno]{amsmath}
\usepackage{amsthm,amssymb}
\usepackage[mathcal]{euscript}
\usepackage{enumitem}
\usepackage{tikz}
\usepackage[unicode=true]{hyperref}
\usepackage{bookmark}

\hypersetup{colorlinks=true,linkcolor=black,citecolor=black,filecolor=black,urlcolor=black}

\makeatletter
\let\citationorig\citation
\def\citation#1{\citationorig{#1}\@for\@tempa:=#1\do{\@ifundefined{cit@\@tempa}{\global\@namedef{cit@\@tempa}{}}{}}}
\let\bibitemorig\bibitem
\def\bibitem#1{\@ifundefined{cit@#1}{\typeout{LaTeX Warning: Unused bibitem `#1'}}{}\bibitemorig{#1}}
\let\old@setaddresses\@setaddresses
\def\@setaddresses{\bigskip{\parindent 0pt\let\scshape\relax\let\ttfamily\relax\old@setaddresses}}
\makeatother

\renewenvironment{enumerate}{\begin{enumorig}[label=\textup{(\arabic*)}, noitemsep, topsep=\medskipamount]}{\end{enumorig}}
\newenvironment{enumerate'}{\begin{enumorig}[label=\textup{(\arabic*$'$)}, noitemsep, topsep=\medskipamount]}{\end{enumorig}}
\newenvironment{enumerate''}{\begin{enumorig}[label=\textup{(\arabic*$''$)}, noitemsep, topsep=\medskipamount]}{\end{enumorig}}
\newenvironment{enumeratea}{\begin{enumorig}[label=\textup{(\alph*)}, noitemsep, topsep=\medskipamount]}{\end{enumorig}}

\renewenvironment{itemize}{\begin{itemorig}[label=\textup\textbullet, noitemsep, topsep=\medskipamount, labelsep=.6em]}{\end{itemorig}}

\newtheorem{theorem}{Theorem}[section]
\newtheorem{proposition}[theorem]{Proposition}
\newtheorem{lemma}[theorem]{Lemma}
\newtheorem{corollary}[theorem]{Corollary}
\newtheorem{observation}[theorem]{Observation}
\theoremstyle{remark}
\newtheorem{example}[theorem]{Example}
\numberwithin{equation}{section}

\def\setN{\mathbb{N}}
\def\calB{\mathcal{B}}
\def\calC{\mathcal{C}}
\def\calG{\mathcal{G}}
\def\calM{\mathcal{M}}
\def\calS{\mathcal{S}}

\def\size#1{\lvert#1\rvert}
\def\link#1#2{#1\{#2\}}
\def\balance{w^\star}

\makeatletter
\newcommand{\superimpose}[2]{\ooalign{$#1\@firstoftwo#2$\cr\hfil$#1\@secondoftwo#2$\hfil\cr}}
\makeatother

\def\widecup{\mathbin{\mathpalette\superimpose{{\hphantom{+}}{\cup}}}}
\def\tbigcup{\mathop{\textstyle\bigcup}\nolimits}

\let\leq\leqslant
\let\geq\geqslant
\let\setminus\smallsetminus
\let\subset\subseteq
\let\supset\supseteq
\let\Theta\varTheta
\let\Omega\varOmega

\allowdisplaybreaks

\hypersetup{
  pdftitle={Graph sharing game and the structure of weighted graphs with a forbidden subdivision},
  pdfauthor={Adam Gągol, Piotr Micek, Bartosz Walczak}
}

\title[Graph sharing game and graphs with a forbidden subdivision]{Graph sharing game and the structure of weighted graphs with a forbidden subdivision}

\author{Adam Gągol\and Piotr Micek\and Bartosz Walczak}

\address{Department of Theoretical Computer Science, Faculty of Mathematics and Computer Science, Jagiellonian University, Kraków, Poland}
\email{\href{mailto:gagol@tcs.uj.edu.pl}{gagol@tcs.uj.edu.pl}, \href{mailto:micek@tcs.uj.edu.pl}{micek@tcs.uj.edu.pl}, \href{mailto:walczak@tcs.uj.edu.pl}{walczak@tcs.uj.edu.pl}}

\thanks{A journal version of this paper appeared in \href{http://doi.org/10.1002/jgt.22045}{\emph{J. Graph Theory} 85~(1), 22--50, 2017}.}
\thanks{Adam Gągol and Bartosz Walczak were partially supported by National Science Center of Poland under grant 2011/03/N/ST6/03111.}

\begin{document}

\settowidth\leftmargini{(1$''$)\hskip\labelsep}

\begin{abstract}
In the \emph{graph sharing game}, two players share a connected graph $G$ with non-negative weights assigned to the vertices, claiming and collecting the vertices of $G$ one by one, while keeping the set of all claimed vertices connected through the whole game.
Each player wants to maximize the total weight of the vertices they have gathered by the end of the game, when the whole $G$ has been claimed.
It is proved that for any class $\mathcal{G}$ of graphs with an odd number of vertices and with forbidden subdivision of a fixed graph (e.g., for the class $\mathcal{G}$ of planar graphs with an odd number of vertices), there is a constant $c_{\mathcal{G}}>0$ such that the first player can secure at least the $c_{\mathcal{G}}$ proportion of the total weight of $G$ whenever $G\in\mathcal{G}$.
Known examples show that such a constant does no longer exist if any of the two conditions on the class $\mathcal{G}$ (an odd number of vertices or a forbidden subdivision) is removed.
The main ingredient in the proof is a new structural result on weighted graphs with a forbidden subdivision.
\end{abstract}

\maketitle

\section{Introduction}

The \emph{graph sharing game} is played by two players, Alice and Bob, on a connected graph $G$ with non-negative weights assigned to the vertices.
Starting with Alice, the players alternate in taking the vertices of $G$ one by one until the whole $G$ has been taken.
The rule is that the set of all taken vertices must induce a connected subgraph of $G$ through the whole game.
Each player wants to maximize the total weight of the vertices they have gathered by the end.

The above is one of the two graph sharing games introduced by Cibulka, Kynčl, Mészáros, Stolař, and Valtr \cite{CKM+13} and independently by Micek and Walczak \cite{MiW11,MiW12}.
They called it the graph sharing game with taken part connected or \emph{game T}\@.
The other game, called the graph sharing game with remaining part connected or \emph{game R}, differs from game T in that the remaining (non-taken) part of the graph must be connected instead of the taken part.
Both games originate from the ``pizza sharing game'' popularized by Peter Winkler, which is either game T or game R played on a cycle.
Cibulka et~al.\ \cite{CKM+10} and independently Knauer, Micek, and Ueckerdt \cite{KMU11} proved that Alice has a strategy to collect at least $4/9$ of the total weight of any pizza.
This bound is best possible (see Figure \ref{fig:4/9}).

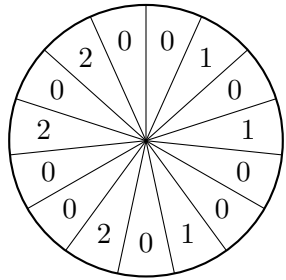
\begin{figure}[t]
\centering
\begin{tikzpicture}[scale=1.8]
  \draw[thick] (0,0) circle (1cm);
  \draw (0,0) -- ( 90:1cm) (102:7.5mm) node {$0$};
  \draw (0,0) -- (114:1cm) (126:7.5mm) node {$2$};
  \draw (0,0) -- (138:1cm) (150:7.5mm) node {$0$};
  \draw (0,0) -- (162:1cm) (174:7.5mm) node {$2$};
  \draw (0,0) -- (186:1cm) (198:7.5mm) node {$0$};
  \draw (0,0) -- (210:1cm) (222:7.5mm) node {$0$};
  \draw (0,0) -- (234:1cm) (246:7.5mm) node {$2$};
  \draw (0,0) -- (258:1cm) (270:7.5mm) node {$0$};
  \draw (0,0) -- (282:1cm) (294:7.5mm) node {$1$};
  \draw (0,0) -- (306:1cm) (318:7.5mm) node {$0$};
  \draw (0,0) -- (330:1cm) (342:7.5mm) node {$0$};
  \draw (0,0) -- (354:1cm) (  6:7.5mm) node {$1$};
  \draw (0,0) -- ( 18:1cm) ( 30:7.5mm) node {$0$};
  \draw (0,0) -- ( 42:1cm) ( 54:7.5mm) node {$1$};
  \draw (0,0) -- ( 66:1cm) ( 78:7.5mm) node {$0$};
\end{tikzpicture}
\caption{Alice can get at most $\frac{4}{9}$ of the pizza above playing against clever Bob.
Numbers stand for slice weights.}
\label{fig:4/9}
\end{figure}

Easy examples show that there is no hope in obtaining a similar result for either game T or game R on general graphs \cite{CKM+13,MiW11,MiW12}.
However, an appropriate restriction on the parity of the number of vertices and the structure of the graph can yield existence of good strategies of Alice.

For game T, Micek and Walczak \cite{MiW12} constructed very simple graphs (caterpillars and subdivided stars) with an even number of vertices and with arbitrarily small guaranteed outcome of Alice (see Example \ref{ex:hedgehog}).
On the other hand, they proved that Alice can always secure at least $1/4$ of the total weight playing on a tree with an odd number of vertices.
They also constructed a family of graphs with an odd number of vertices that are arbitrarily bad for Alice.
These graphs contain subdivisions of arbitrarily large cliques.
As the main result of this paper, we prove that these subdivisions are unavoidable.

\begin{theorem}
\label{thm:intro-game}
For every positive integer\/ $n$, there is\/ $c_n\in(0,1]$ such that if\/ $G$ is a weighted connected graph with an odd number of vertices and with no subdivision of\/ $K_n$, then Alice has a strategy to collect vertices of total weight at least\/ $c_nw(G)$ in the graph sharing game on\/ $G$ with taken part connected, where\/ $w(G)$ denotes the total weight of\/ $G$.
\end{theorem}

Our proof of Theorem \ref{thm:intro-game} gives $c_n=\Omega(2^{-p(n)})$ for some polynomial $p$, and we make no effort to optimize this bound.
Gągol \cite{Gag-master} proved that if $G$ is a connected graph with an odd number of vertices, with weights $0$ or $1$ on the vertices, and with no $K_n$ minor, then Alice can guarantee herself at least $\Omega(n^{-3}\log^{-3/2}n)$ of the total weight in game T on $G$.
His argument is tailored for that special case and is unlikely to generalize.
Still, we expect the optimum value of $c_n$ in Theorem \ref{thm:intro-game} to be of order $\Omega(n^{-\alpha})$ for some $\alpha\geq 1$.
The optimum value of $c_3$, which is the maximum proportion of the total weight Alice can secure on odd trees, lies between $1/4$ and $2/5$ \cite{MiW12}.

We also present an example (Example \ref{ex:bounded-exp}) illustrating that the forbidden subdivision condition in Theorem \ref{thm:intro-game} cannot be replaced by bounded expansion, which is the next restriction on a class of graphs (weaker than that of a forbidden subdivision) in the taxonomy of sparse graph classes due to Nešetřil and Ossona de Mendez \cite{NeO-book}.

The main ingredient in the proof of Theorem \ref{thm:intro-game} is the following structural result, which may be of independent interest.

\begin{theorem}
\label{thm:intro-struct}
For every positive integer\/ $n$, there is\/ $c_n\in(0,1]$ such that if\/ $G$ is a weighted connected graph with no subdivision of\/ $K_n$, then at least one of the following holds:
\begin{enumerate}
\item There is a connected set\/ $S\subset V(G)$ such that the total weight of all components of\/ $G\setminus S$ other than the heaviest one is at least\/ $c_nw(G)$.
\item There are a set\/ $S\subset V(G)$ with\/ $w(S)\geq c_nw(G)$ and a cyclic ordering of\/ $S$ such that the neighborhood of every component of\/ $G\setminus S$ consists of either a single vertex in\/ $S$ or two vertices in\/ $S$ consecutive in the cyclic order.
\end{enumerate}
In the above, $w(G)$ and\/ $w(S)$ denote the total weights of\/ $G$ and\/ $S$, respectively.
\end{theorem}

Our proof of Theorem \ref{thm:intro-struct} gives $c_n=\Omega(2^{-\smash{\binom{n}{2}}}n^{-1})$, and again we make no effort to optimize this bound.

For game R, the two parities of the number of vertices switch their roles.
Even very simple graphs with an odd number of vertices (like a $3$-vertex path with all the weight in the middle) can be very bad for Alice.
There also exist graphs with an even number of vertices and with arbitrarily small guaranteed outcome of Alice \cite{CKM+13,MiW11}.
On the other hand, Micek and Walczak \cite{MiW11} proved that Alice can secure $1/4$ of the total weight in game R played on a tree with an even number of vertices, and they conjectured that she can do as much as $1/2$.
This was proved by Seacrest and Seacrest \cite{SeS12}, who also conjectured that Alice can secure some positive constant proportion of the total weight in game R on all bipartite graphs with an even number of vertices.
However, no result of this kind is known for any natural class of graphs broader than the class of trees with an even number of vertices.

For both variants of the game, Cibulka et~al.\ \cite{CKM+13} described other constructions of graphs that are very bad for Alice, imposing various restrictions on the structure of the graph (e.g., $k$-connectivity).
They also investigated computational aspects of the game.
They proved that finding an optimal strategy in game R is PSPACE-complete in general.
Whether the same is true for game T is open.
Cibulka et~al.\ \cite{CKM+13} also asked about the complexity of finding an optimal strategy in games T and R on trees.
A polynomial-time algorithm for game T on trees was devised by Walczak \cite{Wal-phd}.
The problem for game R on trees remains open.

For the rest of the paper, we focus only on game T, which we simply call the graph sharing game, as it is defined in the first paragraph.
After setting up some graph-theoretic background in Section \ref{sec:background}, we review constructions of graphs with arbitrarily small guaranteed outcome of Alice in Section \ref{sec:examples}.
We prove Theorem \ref{thm:intro-struct} in Section \ref{sec:structural} (see Corollary \ref{cor:struct-subdiv}) and Theorem \ref{thm:intro-game} in Section \ref{sec:strategies}.

\section{Background}
\label{sec:background}

\subsection{Basic terminology and notation}

We let $\setN$ and $\setN^+$ denote the sets of non-negative integers and of positive integers, respectively.
We assume that the reader is familiar with basic terminology of graph theory.
Every graph that we consider is finite and has no loops or multiple edges.
The sets of vertices and edges of a graph $G$ are denoted by $V(G)$ and $E(G)$, respectively.
For a vertex $v\in V(G)$, we define
\begin{itemize}
\item $N_G(v)$ as the neighborhood of $v$ in $G$,
\item $N_G[v]$ as the closed neighborhood of $v$ in $G$, that is, the set $N_G(v)\cup\{v\}$.
\end{itemize}
For a set $S\subset V(G)$, we define
\begin{itemize}
\item $N_G(S)$ as the neighborhood of $S$ in $G$, that is, the set of vertices in $V(G)\setminus S$ adjacent to at least one vertex in $S$,
\item $N_G[S]$ as the closed neighborhood of $S$ in $G$, that is, the set $N_G(S)\cup S$,
\item $G[S]$ as the subgraph of $G$ induced on $S$,
\item $G\cap S=G[V(G)\cap S]$ and $G\setminus S=G[V(G)\setminus S]$.
\end{itemize}
We omit the subscript $G$ in $N_G$ when the graph $G$ is clear from the context.

A \emph{weighted graph} is a graph $G$ equipped with a function $w_G\colon V(G)\to[0,\infty)$ that assigns a \emph{weight} to each vertex of $G$.
If $S\subset V(G)$, then $w_G(S)$ denotes the sum of the weights of the vertices in $S$.
We omit the subscript $G$ in $w_G$ when the graph $G$ is clear from the context.
We define $w(G)=w_G(V(G))$.

A \emph{component} of $G$ is a maximal connected subgraph of $G$.
The family of components of $G$ is denoted by $\calC(G)$.
The sets $V(C)$ over all $C\in\calC(G)$ form a partition of $V(G)$.
For a weighted graph $G$, we define
\begin{equation*}
\balance(G)=w(G)-\max_{C\in\calC(G)}w(C).
\end{equation*}
The quantity $\balance(G)$ is essential for our considerations in Sections \ref{sec:structural} and \ref{sec:strategies}.
It is large unless most of the weight of $G$ lies in only one component of $G$.
When $G$ is connected, $\balance(G)=0$.

If $\calS$ is a partition of $V(G)$ into non-empty connected subsets, then $G/\calS$ denotes the graph with vertex set $\calS$ and edge set defined as follows: $XY\in E(G/\calS)$ if and only if there are $x\in X$ and $y\in Y$ such that $xy\in E(G)$.

The \emph{reduction} of a weighted graph $G$ to a set $S\subset V(G)$, denoted by $\link{G}{S}$, is the graph with vertex set $S$, edge set defined so that $uv\in E(\link{G}{S})$ if and only if there is a path $P$ in $G$ connecting $u$ and $v$ internally disjoint from $S$ (such that $V(P)\cap S=\{u,v\}$), and weight function $w_G$ restricted to $S$.
See Figure \ref{fig:link} for an illustration.
If $G$ is connected, then $\link{G}{S}$ is connected for every $S$.

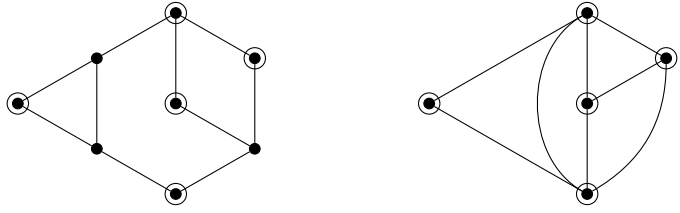
\begin{figure}[t]
\centering
\begin{tikzpicture}[scale=1.2]
  \tikzstyle{every node}=[circle,draw,minimum size=8pt,inner sep=0pt]
  \draw (  0,0  ) node {}
        (  0,1  ) node {}
        ( 30:1cm) node {}
        (  0,-1 ) node {}
       +(150:2cm) node {};
  \tikzstyle{every node}=[circle,draw,fill,minimum size=4pt,inner sep=0pt]
  \draw (   0,0  ) node (a) {}
        (   0,1  ) node (b) {}
        (  30:1cm) node (c) {}
        (   0,-1 ) node (d) {}
       +( 150:2cm) node (e) {}
        ( -30:1cm) node (f) {}
        (-150:1cm) node (g) {}
        ( 150:1cm) node (h) {};
  \path (a) edge (b) edge (f)
        (b) edge (c) edge (h)
        (c) edge (f)
        (d) edge (f) edge (g)
        (e) edge (g) edge (h)
        (g) edge (h);
\end{tikzpicture}\hskip 2cm
\begin{tikzpicture}[scale=1.2]
  \tikzstyle{every node}=[circle,draw,minimum size=8pt,inner sep=0pt]
  \draw (  0,0  ) node {}
        (  0,1  ) node {}
        ( 30:1cm) node {}
        (  0,-1 ) node {}
       +(150:2cm) node {};
  \tikzstyle{every node}=[circle,draw,fill,minimum size=4pt,inner sep=0pt]
  \draw (   0,0  ) node (a) {}
        (   0,1  ) node (b) {}
        (  30:1cm) node (c) {}
        (   0,-1 ) node (d) {}
       +( 150:2cm) node (e) {};
  \path (a) edge (b) edge (c) edge (d)
        (b) edge (c) edge[bend right=60] (d) edge (e)
        (c) edge[bend left=30] (d)
        (d) edge (e);
\end{tikzpicture}
\caption{Left: a graph $G$, a set $S$ of vertices circled; right: $\link{G}{S}$}
\label{fig:link}
\end{figure}

We use a special definition of \emph{cycles}, which differs from the standard one as follows: a graph consisting of a single vertex is a cycle of length $1$, while a graph with two vertices joined by an edge is a cycle of length $2$.
This saves us from considering special degenerate cases in Sections \ref{sec:structural} and \ref{sec:strategies}.

\subsection{Subdivisions and shallow minors}
\label{sec:subdiv-minors}

To \emph{subdivide} an edge $uv$ of a graph $H$ is to replace the edge $uv$ by a path $P_{uv}$ between $u$ and $v$ of any length passing through new vertices that do not belong to $V(H)$.
A graph $G$ is a \emph{subdivision} of a graph $H$ if $G$ arises from $H$ by subdividing edges, that is, replacing edges $uv\in E(H)$ by paths $P_{uv}$ that are internally disjoint from $V(H)$ and from each other.
The relation of being a subdivision is transitive, that is, if $F$ is a subdivision of $G$ and $G$ is a subdivision of $H$ then $F$ is also a subdivision of $H$.
We say that a graph $G$ contains a subdivision of $K_n$ if $G$ has a subgraph that is a subdivision of the complete graph on $n$ vertices.

\begin{theorem}[Mader {\cite[Satz 2]{Mad67}}]
\label{thm:mader}
For every\/ $n\in\setN^+$, there is\/ $d_n\in\setN$ such that every graph\/ $G$ with no subdivision of\/ $K_n$ has a vertex of degree at most\/ $d_n$.
\end{theorem}

Komlós and Szemerédi \cite{KoS96} and independently Bollobás and Thomason \cite{BoT98} proved that the above holds with $d_n=O(n^2)$.
Jung \cite{Jun67} observed that the complete bipartite graph $K_{m,m}$ contains no subdivision of $K_n$ unless $m=\Omega(n^2)$.

A set $S\subset V(G)$ is \emph{$r$-shallow} if there is $v\in S$ such that every vertex in $S$ is within distance at most $r$ from $v$ in $G[S]$.
Every $r$-shallow subset of $V(G)$ is non-empty and connected.
A graph $H$ is an \emph{$r$-shallow minor} of $G$ if $H$ is isomorphic to a subgraph of $G/\calS$ for some partition $\calS$ of $V(G)$ into $r$-shallow subsets.
Equivalently, $H$ is an $r$-shallow minor of $G$ if there is a family $\{S(v)\}_{v\in V(H)}$ of pairwise disjoint $r$-shallow subsets of $V(G)$ such that there is an edge between $S(u)$ and $S(v)$ in $G$ whenever $uv\in E(H)$.
A graph $H$ is a \emph{minor} of $G$ if $H$ is an $r$-shallow minor of $G$ for some $r\in\setN$.
If $G$ contains a subdivision of a graph isomorphic to $H$, then $H$ is a minor of $G$ (such a graph $H$ is also called a \emph{topological minor} of $G$).

\begin{theorem}[Nešetřil, Ossona de Mendez {\cite[Lemma 9.8]{NeO06}}]
\label{thm:shallow-subdiv}
For every\/ $n\in\setN^+$, there is\/ $N\in\setN^+$ such that if\/ $G$ is a graph with no subdivision of\/ $K_n$, then every\/ $1$-shallow minor of\/ $G$ contains no subdivision of\/ $K_N$.
\end{theorem}

Nešetřil and Ossona de Mendez proved the above with $N=2n^2-6n+8$.

The \emph{greatest reduced average degree} (\emph{grad}) of \emph{rank} $r$ of a graph $G$, denoted by $\nabla_r(G)$, is defined by
\begin{equation*}
\nabla_r(G)=\max_{H\in\calM_r(G)}\frac{\size{E(H)}}{\size{V(H)}},
\end{equation*}
where $\calM_r(G)$ denotes the class of all $r$-shallow minors of $G$.
It is clear that $\nabla_r(G)\leq\nabla_s(G)$ when $r\leq s$.
Theorem \ref{thm:mader} has the following equivalent formulation: for every $n\in\setN^+$, the graphs $G$ with no subdivision of $K_n$ have bounded $\nabla_0(G)$.
Since an $r$-shallow minor of $G$ can be obtained from $G$ by iteratively making a $1$-shallow minor $O(\log r)$ times, the above and Theorem \ref{thm:shallow-subdiv} imply that for any $n\in\setN^+$ and $r\in\setN$, the graphs $G$ with no subdivision of $K_n$ have bounded $\nabla_r(G)$.
A class of graphs $\calG$ has \emph{bounded expansion} if there is a function $f\colon\setN\to[0,\infty)$ such that $\nabla_r(G)\leq f(r)$ for any $G\in\calG$ and $r\in\setN$.
Hence, for every $n\in\setN^+$, the class of graphs with no subdivision of $K_n$ has bounded expansion.

\subsection{Arrangeability}
\label{subsec:arrangeable}

For a linear ordering $\pi$ of a set $S$, we define
\begin{equation*}
\pi^-(x)=\bigl\{y\in S\colon y<_\pi x\bigr\},\qquad\pi^+(x)=\bigl\{y\in S\colon y>_\pi x\bigr\}.
\end{equation*}
For a graph $G$, a linear ordering $\pi$ of $V(G)$, $v\in V(G)$, and $S\subset V(G)$, we define
\begin{gather*}
N_\pi^-(v)=N_G(v)\cap\pi^-(v),\qquad N_\pi^+(v)=N_G(v)\cap\pi^+(v),\\
N_\pi^-(S)=\tbigcup_{v\in S}N_\pi^-(v).
\end{gather*}
A graph $G$ is \emph{$p$-arrangeable} if there is a linear ordering $\pi$ of $V(G)$ with the following property: for every $v\in V(G)$, we have $\size{N_\pi^-(N_\pi^+(v))\cap\pi^-(v)}\leq p$.
The \emph{arrangeability} of $G$ is the minimum $p$ such that $G$ is $p$-arrangeable.
This parameter has been introduced by Chen and Schelp \cite{ChS93}, who proved that graphs with bounded arrangeability have linearly bounded Ramsey numbers.

\begin{observation}
\label{obs:arrangeable}
If\/ $G$ is a\/ $p$-arrangeable graph and\/ $\pi$ is a linear ordering of\/ $V(G)$ that satisfies the condition of the definition of\/ $p$-arrangeability above, then
\begin{enumerate}
\item\label{item:arrangeable-1} for every\/ $v\in V(G)$, we have\/ $\size{N_\pi^-(v)}\leq p+1$,
\item\label{item:arrangeable-2} for every\/ $v\in V(G)$, we have\/ $\size{N_\pi^-(N[v])\cap\pi^-(v)}\leq p^2+4p+2$.
\end{enumerate}
\end{observation}

\begin{proof}
Let $v\in V(G)$ and suppose $N_\pi^-(v)\neq\emptyset$.
Let $v'$ be the vertex in $N_\pi^-(v)$ that is greatest with respect to $\pi$.
The condition on $\pi$ applied to $v'$ yields $\size{N_\pi^-(v)\cap\pi^-(v')}\leq p$, which implies \ref{item:arrangeable-1}.
Now, for every $v\in V(G)$, we have
\begin{gather*}
\begin{alignedat}{7}
&N_\pi^-(N[v])\cap\pi^-(v)&&=&&N_\pi^-(N_\pi^-(v))&&\widecup{}&&N_\pi^-(v)&&\widecup{}&\mskip-1mu\bigl(&N_\pi^-(N_\pi^+(v))\cap\pi^-(v)\bigr),\\
\lvert&N_\pi^-(N[v])\cap\pi^-(v)\rvert&&\leq&\lvert&N_\pi^-(N_\pi^-(v))\rvert&&+{}&\lvert&N_\pi^-(v)\rvert&&+{}&\lvert&N_\pi^-(N_\pi^+(v))\cap\pi^-(v)\rvert\\
&&&\leq\hbox to 0pt{$(p+1)^2+(p+1)+p=p^2+4p+2,$\hss}
\end{alignedat}
\end{gather*}
where the first term is bounded using an iteration of \ref{item:arrangeable-1}.
This proves \ref{item:arrangeable-2}.
\end{proof}

\begin{theorem}[Rödl, Thomas \cite{RoT97}]
\label{thm:arrangeable}
For every\/ $n\in\setN^+$, there is\/ $p_n\in\setN$ such that every graph with no subdivision of\/ $K_n$ is\/ $p_n$-arrangeable.
\end{theorem}

Rödl and Thomas proved the above with $p_n=O(n^8)$.
This was improved by Dvořák \cite{Dvo08} to $O(n^6)$.
Nešetřil and Ossona de Mendez \cite[Theorem 4.4]{NeO09} proved that the arrangeability of a graph $G$ is bounded in terms of $\nabla_1(G)$.
In particular, arrangeability is bounded in classes of graphs with bounded expansion.

\section{Examples}
\label{sec:examples}

First, we recall some known constructions of weighted connected graphs on which Alice's guaranteed outcome in the graph sharing game can be an arbitrarily small proportion of the total weight.
Then, we show how to modify such constructions so as to obtain graphs with the same property that are very sparse (in particular, have bounded expansion).

\begin{example}[\cite{MiW12}]
\label{ex:hedgehog}
Let $G_n$ be a weighted graph with vertex set $\{a_1,\ldots,a_n$, $b_1,\ldots,b_n\}$ such that the subgraph of $G_n$ induced on $\{b_1,\ldots,b_n\}$ is connected and the only neighbor of each $a_i$ is $b_i$.
Each $a_i$ has weight $1$, and each $b_i$ has weight $0$.
The total weight is $n$.
See Figure \ref{fig:hedgehog} for an illustration.

\begin{figure}[t]
\centering
\begin{tikzpicture}[scale=0.851]
  \tikzstyle{every node}=[circle,draw,inner sep=3pt]
  \node (a1) at ( 90:1cm    ) {$0$};
  \node (b1) at ( 90:2.175cm) {$1$};
  \node (a2) at (162:1cm    ) {$0$};
  \node (b2) at (162:2.175cm) {$1$};
  \node (a3) at (234:1cm    ) {$0$};
  \node (b3) at (234:2.175cm) {$1$};
  \node (a4) at (306:1cm    ) {$0$};
  \node (b4) at (306:2.175cm) {$1$};
  \node (a5) at ( 18:1cm    ) {$0$};
  \node (b5) at ( 18:2.175cm) {$1$};
  \path (a1) edge (b1) edge (a2) edge (a5)
        (a2) edge (b2) edge (a3)
        (a3) edge (b3)
        (a4) edge (b4) edge (a5)
        (a5) edge (b5);
\end{tikzpicture}\hspace{1.5cm}
\begin{tikzpicture}[scale=0.851]
  \tikzstyle{every node}=[circle,draw,inner sep=3pt]
  \node (a1) at ( 90:1cm    ) {$0$};
  \node (b1) at ( 90:2.175cm) {$1$};
  \node (a2) at (162:1cm    ) {$0$};
  \node (b2) at (162:2.175cm) {$1$};
  \node (a3) at (234:1cm    ) {$0$};
  \node (b3) at (234:2.175cm) {$1$};
  \node (a4) at (306:1cm    ) {$0$};
  \node (b4) at (306:2.175cm) {$1$};
  \node (a5) at ( 18:1cm    ) {$0$};
  \node (b5) at ( 18:2.175cm) {$1$};
  \path (a1) edge (b1) edge (a2) edge (a3) edge (a4) edge (a5)
        (a2) edge (b2)
        (a3) edge (b3)
        (a4) edge (b4)
        (a5) edge (b5);
\end{tikzpicture}
\caption{Examples of $G_5$}
\label{fig:hedgehog}
\end{figure}
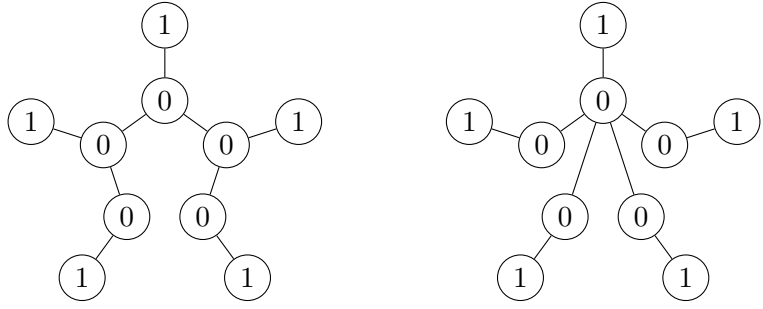

Alice has no strategy to gather more than $1$ from $G_n$.
Indeed, she starts with some $a_i$ (collecting $1$) or $b_i$, and clever Bob responds by taking the other of $a_i$, $b_i$.
In all subsequent moves Alice is forced to take some vertex of $b_1,\ldots,b_n$, say $b_j$, and Bob responds by playing $a_j$.
\end{example}

\begin{example}[\cite{MiW12}]
\label{ex:odd-construction}
Let $H_n$ be a weighted graph with vertex set $\{a_1,\ldots,a_n$, $b_1,\ldots,b_n\}\cup\{c_X\colon X\subset\{1,\ldots,n\}$ and $X\neq\emptyset\}$, which has size $2n+2^n-1$.
The neighborhoods of the vertices are $N(a_i)=\{b_i\}$, $N(b_i)=\{a_i\}\cup\{c_X\colon i\in X\}$, and $N(c_X)=\{b_i\colon i\in X\}$.
Each $a_i$ has weight $1$, and all other vertices have weight $0$.
The total weight is again $n$.
See Figure \ref{fig:odd-construction} for an illustration.

\begin{figure}
\centering
\begin{tikzpicture}[scale=1,xscale=2]
  \tikzstyle{every node}=[circle,draw,inner sep=3pt]
  \tikzstyle{every label}=[rectangle,draw=none]
  \node[label=right:$a_1$]           (a1)   at (   0,0) {$1$};
  \node[label=right:$a_2$]           (a2)   at (   1,0) {$1$};
  \node[label=right:$a_3$]           (a3)   at (   2,0) {$1$};
  \node[label=right:$b_1$]           (b1)   at (   0,1) {$0$};
  \node[label=right:$b_2$]           (b2)   at (   1,1) {$0$};
  \node[label=right:$b_3$]           (b3)   at (   2,1) {$0$};
  \node[label=left:$c_{\{1\}}$]      (c1)   at (   0,2) {$0$};
  \node                              (c2)   at (   1,2) {$0$};
  \node[label=right:$c_{\{3\}}$]     (c3)   at (   2,2) {$0$};
  \node[label=left:$c_{\{1,2\}}$]    (c12)  at (0.42,3) {$0$};
  \node                              (c13)  at (   1,3) {$0$};
  \node[label=right:$c_{\{2,3\}}$]   (c23)  at (1.58,3) {$0$};
  \node[label=right:$c_{\{1,2,3\}}$] (c123) at (   1,4) {$0$};
  \path (b1) edge (a1) edge (c1) edge (c12) edge (c13) edge (c123);
  \path (b2) edge (a2) edge (c2) edge (c12) edge (c23) edge[bend left=14] (c123);
  \path (b3) edge (a3) edge (c3) edge (c13) edge (c23) edge (c123);
\end{tikzpicture}
\caption{$H_3$;\enspace to obtain $H_3'$, subdivide each edge $b_ic_X$ by a large even number of new vertices of zero weight}
\label{fig:odd-construction}
\end{figure}
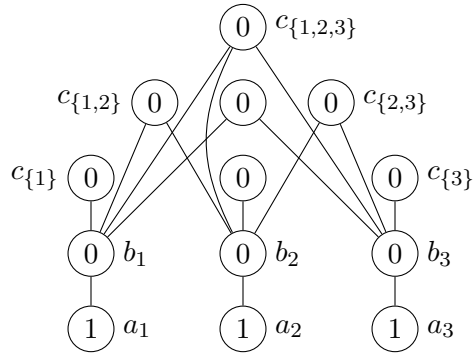

Again, Alice has no strategy to gather more than $1$ from $H_n$.
Let $V=V(H_n)$.
Suppose that Alice starts with $a_{i_1}$ or $b_{i_1}$.
Bob responds by taking the other of $a_{i_1}$, $b_{i_1}$.
Let $V_1=\{a_j,b_j\colon j\neq i_1\}\cup\{c_X\colon i_1\notin X\}$.
If $n-1>0$, then the subgraph induced on $V_1$ is isomorphic to $H_{n-1}$.
In particular, $\size{V_1}$ is odd and $\size{V\setminus V_1}$ is even.
Since $b_{i_1}$ has been taken, all vertices in $V\setminus V_1$ are available.
Therefore, as long as Alice plays in $V\setminus V_1$, Bob can respond also in $V\setminus V_1$.
Alice is eventually forced to enter $V_1$, which is possible only by taking some $b_{i_2}$, and Bob immediately follows with $a_{i_2}$.
If $n-2>0$, then we define $V_2=\{a_j,b_j\colon j\neq i_1,i_2\}\cup\{c_X\colon i_1,i_2\notin X\}$ and continue with the same argument, and so on.
This way Bob wins all of $a_1,\ldots,a_n$ except $a_{i_1}$.
If Alice starts with some $c_X$, then Bob takes any available $b_{i_1}$, and the same argument shows that Bob can take all of $a_1,\ldots,a_n$ except $a_{i_1}$.
\end{example}

Example \ref{ex:hedgehog} shows that very simple trees (caterpillars or subdivided stars) with an even number of vertices can be arbitrarily bad for Alice.
Example \ref{ex:odd-construction} shows that there are also (quite dense) graphs with an odd number of vertices that are arbitrarily bad for Alice.
In particular, the graph $H_n$ contains a subdivision of $K_n$ in which every edge is subdivided by one new vertex.
We can obtain much sparser examples with the help of the following proposition.

\begin{proposition}
\label{prop:subdivision}
Let\/ $G$ be a weighted connected graph such that any two vertices with positive weight are at distance at least\/ $3$ in\/ $G$.
If\/ $H$ is a graph obtained from\/ $G$ by subdividing every edge\/ $uv$ with\/ $w(u)=w(v)=0$ by an even number of vertices of zero weight, then the guaranteed outcome of Alice on\/ $H$ is equal to the one on\/ $G$.
\end{proposition}

\begin{proof}[Proof sketch]
Since the distance between any two vertices with positive weight is at least $3$, taking a vertex with positive weight makes no other vertex with positive weight available, and taking a vertex with zero weight makes at most one vertex with positive weight available.
It follows that an optimal strategy of either player satisfies the following: whenever a vertex with positive weight is available, take it; otherwise, if possible, take a vertex with zero weight that makes no vertex with positive weight available.
Therefore, when both players play optimally, the game on $H$ differs from the game on $G$ only by an additional even number of vertices of zero weight taken between some pairs of vertices with positive weight.
\end{proof}

Proposition \ref{prop:subdivision} allows us to turn the graphs from Example \ref{ex:odd-construction} (or any other family of graphs with similar properties) into a family of much sparser graphs with an odd number of vertices and with the same guaranteed outcome of Alice.

\begin{example}
\label{ex:bounded-exp}
Fix a non-decreasing function $f\colon\setN\to(1,\infty)$ with $f(r)\to\infty$ as $r\to\infty$.
Let $H_n$ for $n\in\setN^+$ denote the graphs constructed in Example \ref{ex:odd-construction}.
Let $z_n=\max_{r\in\setN}\nabla_r(H_n)$.
This is a finite number, as $H_n$ has a bounded number of minors.
It follows that $\nabla_r(H_n')\leq z_n$ for every subdivision $H_n'$ of $H_n$ and every $r\in\setN$.
For any fixed $r\in\setN$, if we subdivide each edge $b_ic_X$ of $H_n$ by a large enough number of new vertices, thus obtaining a graph $H_n'$, then the ratio of the number of vertices of degree $2$ to the number of all vertices in every $r$-shallow minor of $H_n'$ is high enough to guarantee $\nabla_r(H_n')\leq f(r)$.
Let $H_n'$ be a graph obtained from $H_n$ by subdividing every edge $b_ic_X$ of $H_n$ by an even number of new vertices large enough to guarantee $\nabla_r(H_n')\leq\min\{z_n,f(r)\}$ for every $r\in\setN$.
It follows that the graphs $H_n$ have expansion bounded by $f$.
Since $\size{V(H_n)}$ is odd, $\size{V(H_n')}$ is odd as well.
The weight of all subdividing vertices is set to zero, so that $w(H_n')=n$.

Alice has no strategy to gather more than $1$ from $H_n'$.
This follows from Proposition \ref{prop:subdivision} or can be proved directly by a modification of the argument for $H_n$ used in Example \ref{ex:odd-construction}.
\end{example}

Proposition \ref{prop:subdivision} and Example \ref{ex:bounded-exp} motivate the assumption that $G$ contains no subdivision of a fixed complete graph in Theorem \ref{thm:intro-game}.
Indeed, if we want to exclude graphs on which Alice's guaranteed outcome is not greater than $1/n$ of the total weight, then we need to exclude at least the graph $H_n$ from Example \ref{ex:odd-construction} and all its subdivisions of the special kind considered in Proposition \ref{prop:subdivision}, including the graphs $H_n'$ from Example \ref{ex:bounded-exp}.

\section{Structural properties of weighted connected graphs}
\label{sec:structural}

Our goal in this section is to prove Theorem \ref{thm:intro-struct}.
Namely, for a suitable constant $c_n>0$, we show that every weighted connected graph $G$ with no subdivision of $K_n$ contains at least one of the following structures:
\begin{itemize}
\item a connected set $S$ of vertices such that $\balance(G\setminus S)\geq c_nw(G)$,
\item a set $S$ of vertices such that $\link{G}{S}$ is a cycle and $w(S)\geq c_nw(G)$,
\end{itemize}
see Corollary \ref{cor:struct-subdiv}.
This provides a base for the strategies of Alice developed in the next section.

First, we show that every weighted connected graph $G$ contains at least one of the structures above or a connected set $S$ of vertices with $w(N(S))\geq c_nw(G)$, see Corollaries \ref{cor:struct-hamil} and \ref{cor:struct-full}.
Then, we reduce the latter case to the first two for graphs with forbidden subdivision of $K_n$.

\subsection{Hamiltonian graphs}

The following lemma deals with \emph{oriented graphs}, that is, graphs in which every edge is assigned an orientation.
An \emph{oriented path} or \emph{cycle} is a path or cycle in which the orientations of edges agree with the order of vertices along the path or cycle.

For an oriented path $P$, we let $<_P$ denote the order of vertices along $P$.
For vertices $u$ and $v$ of an oriented path $P$, we define
\begin{align*}
[u,v]_P&=\bigl\{x\in V(P)\colon u\leq_Px\leq_Pv\bigr\},\\
[u,v)_P&=\bigl\{x\in V(P)\colon u\leq_Px<_Pv\bigr\},\\
(u,v)_P&=\bigl\{x\in V(P)\colon u<_Px<_Pv\bigr\}.
\end{align*}
An oriented path $P$ in an oriented graph $G$ is \emph{Hamiltonian} if $V(P)=V(G)$.
A vertex $x$ of an oriented graph $G$ with a Hamiltonian path $P$ is \emph{$P$-covered} by an edge $uv\in E(G)$ if $x\in(u,v)_P$.
A vertex of such a graph $G$ is \emph{$P$-covered} if it is $P$-covered by at least one edge of $G$.

\begin{lemma}
\label{lem:oriented-path}
Let\/ $G$ be an acyclic oriented graph with at least two vertices and with a Hamiltonian path\/ $P$ starting at\/ $s$ and ending at\/ $t$.
If every vertex in\/ $V(G)\setminus\{s,t\}$ is\/ $P$-covered, then there are two oriented paths\/ $Q_0$ and\/ $Q_1$ in\/ $G$ starting at\/ $s$, ending at\/ $t$, and having no other vertices in common.
\end{lemma}

\begin{proof}
For every $v\in V(G)\setminus\{s,t\}$, since $v$ is $P$-covered, there is an oriented path in $G$ starting at $s$, ending at $t$, and avoiding $v$.
The assertion of the lemma follows by the directed version of Menger's theorem (see, e.g., \cite{Dir66}).
\end{proof}

For vertices $u$ and $v$ of an oriented cycle $H$, we let $(u,v)_H$ denote
\begin{itemize}
\item the set of internal vertices of the path in $H$ from $u$ to $v$ if $u\neq v$,
\item the set $V(H)\setminus\{u\}$ if $u=v$.
\end{itemize}
We also define
\begin{equation*}
[u,v)_H=\{u\}\cup(u,v)_H,\qquad(u,v]_H=(u,v)_H\cup\{v\}.
\end{equation*}
An (oriented) cycle $H$ in an (oriented) graph $G$ is \emph{Hamiltonian} if $V(H)=V(G)$.
A vertex $x$ of an oriented graph $G$ with a Hamiltonian cycle $H$ is \emph{$H$-covered} by an edge $uv\in E(G)$ if $x\in(u,v)_H$.

Recall that we consider a single-vertex graph and a graph consisting of two vertices joined by an edge as (unoriented) cycles of length $1$ and $2$, respectively.

\begin{lemma}
\label{lem:struct-hamil1}
There is\/ $c\in(0,1]$ such that every weighted graph\/ $G$ containing a Hamiltonian cycle\/ $H$ satisfies at least one of the following conditions:
\begin{enumerate}
\item\label{item:struct-hamil1-1} There is a connected set\/ $S\subset V(G)$ such that
\begin{equation*}
\balance(H\setminus S)\geq cw(G).
\end{equation*}
\item\label{item:struct-hamil1-2} There is a set\/ $S\subset V(G)$ such that\/ $\link{G}{S}$ is a cycle and
\begin{equation*}
w(S)\geq cw(G).
\end{equation*}
\end{enumerate}
\end{lemma}

\begin{proof}
We show that it is enough to set
\begin{equation}
\label{eq:struct-hamil1-c}
c=1/5.
\end{equation}

Let $G$ be a weighted graph containing a Hamiltonian cycle $H$.
If $G$ has no more than $3$ vertices, then \ref{item:struct-hamil1-2} holds for $S=V(G)$.
If some edge $uv\in E(G)$ satisfies $\balance(H\setminus\{u,v\})\geq cw(G)$, then \ref{item:struct-hamil1-1} holds for $S=\{u,v\}$.
Therefore, for the remainder of the proof, assume that $G$ has at least $4$ vertices and every edge $uv\in E(G)$ satisfies
\begin{equation}
\label{eq:H-balance}
\balance(H\setminus\{u,v\})<cw(G).
\end{equation}

Suppose that we find two connected sets $S_0,S_1\subset V(G)$ such that
\begin{gather}
\label{eq:S-def-1}
w(S_0\cap S_1)<cw(G),\\
\label{eq:S-def-2}
w(C)<cw(G)\quad\text{for }C\in\calC(H\setminus S_k)\text{ and }k\in\{0,1\}.
\end{gather}
By \eqref{eq:S-def-1}, we have
\begin{equation*}
w(S_0)+w(S_1)=w(S_0\cup S_1)+w(S_0\cap S_1)<(1+c)w(G).
\end{equation*}
Without loss of generality, we have $w(S_0)\leq w(S_1)$, and then the above yields
\begin{equation}
\label{eq:w(S_0)}
w(S_0)<\tfrac{1}{2}(1+c)w(G).
\end{equation}
To conclude, we have
\begin{gather*}
\begin{alignedat}{2}
\balance(H\setminus S_0)&=w(H\setminus S_0)-\max_{C\in\calC(H\setminus S_0)}w(C)\\
&>w(G)-\tfrac{1}{2}(1+c)w(G)-cw(G)&\quad&\text{by \eqref{eq:w(S_0)} and \eqref{eq:S-def-2}}\\
&=cw(G)&\quad&\text{by \eqref{eq:struct-hamil1-c}}.
\end{alignedat}
\end{gather*}
This shows that \ref{item:struct-hamil1-1} holds for $S=S_0$.

To complete the proof, we show how to find two connected sets $S_0,S_1\subset V(G)$ satisfying \eqref{eq:S-def-1} and \eqref{eq:S-def-2} or a set $S\subset V(G)$ satisfying the conclusion \ref{item:struct-hamil1-2} of the lemma.

We orient the edges of $G$ as follows.
First, we orient the cycle $H$ in any of the two directions.
Then, we assign to every edge in $E(G)\setminus E(H)$ an orientation $uv$ so that $w((u,v)_H)\leq w((v,u)_H)$.
This and \eqref{eq:H-balance} imply that every oriented edge $uv\in E(G)\setminus E(H)$ satisfies
\begin{equation}
\label{eq:w(H_uv)}
w((u,v)_H)=\balance(H\setminus\{u,v\})<cw(G).
\end{equation}
From now on, we consider $G$ as an oriented graph and $H$ as an oriented Hamiltonian cycle in $G$.

Let $U$ be the set of vertices of $G$ that are not $H$-covered.
Suppose $U\neq\emptyset$.
Let $u_0,\ldots,u_{n-1}$ be the vertices in $U$ in the order they occur along $H$, and let $u_n=u_0$.
For every edge $xy\in E(G)$, there is an index $i\in\{0,\ldots,n-1\}$ such that $x\in[u_i,u_{i+1})_H$ and $y\in(u_i,u_{i+1}]_H$, as otherwise $xy$ would cover a vertex from $U$.
It follows that $\link{G}{U}$ is the cycle consisting of $u_0,\ldots,u_{n-1}$ in this order.
If $w(U)\geq cw(G)$, then \ref{item:struct-hamil1-2} holds for $S=U$.
Thus assume
\begin{equation}
\label{eq:w(U)}
w(U)<cw(G).
\end{equation}
For every $u_i$, create a new vertex $u'_i$ and redirect all edges of $G$ that end at $u_i$ sending them to $u'_i$.
Thus a new oriented graph $G'$ is obtained.
It splits into $n$ pairwise disjoint acyclic oriented graphs $G'_0,\ldots,G'_{n-1}$, each $G'_i$ containing a Hamiltonian path from $u_i$ to $u'_{i+1}$ going through all vertices in $(u_i,u_{i+1})_H$.
Every vertex $x\in V(G'_i)\setminus\{u_i,u'_{i+1}\}$ is $H$-covered by an edge corresponding to the one that originally $H$-covers $x$ in $G$.
Therefore, by Lemma \ref{lem:oriented-path}, each $G'_i$ has two paths $Q^0_i$ and $Q^1_i$ from $u_i$ to $u'_{i+1}$ containing no other common vertices.
For $k\in\{0,1\}$, let $C_k$ be the cycle in $G$ obtained by taking the union of all $Q^k_i$ and gluing each pair $u_i,u'_i$ back into the single vertex $u_i$.
Let $S_k=V(C_k)$.
It follows that $S_0\cap S_1=U$, and hence \eqref{eq:S-def-1} follows from \eqref{eq:w(U)}.
Moreover, every component of $H\setminus S_k$ is entirely contained in $(u,v)_H$ for some edge $uv\in E(C_k)$, and hence \eqref{eq:S-def-2} follows from \eqref{eq:w(H_uv)}.
Since $S_0$ and $S_1$ are the vertex sets of cycles in $G$, they are connected in $G$.

Now, suppose $U=\emptyset$.
Choose any vertex $v\in V(G)$.
Redirect all edges that $H$-cover $v$ sending them to $v$ (removing duplicates).
Thus a new oriented graph $G^\star$ with the same Hamiltonian cycle $H$ is obtained.
Every redirected edge $uv$ still satisfies \eqref{eq:w(H_uv)}.
The vertex $v$ is not $H$-covered in $G^\star$.
Moreover, all vertices not $H$-covered in $G^\star$ are $H$-covered in $G$ by a common edge.
Therefore, by \eqref{eq:w(H_uv)}, the set $U^\star$ of vertices that are not $H$-covered in $G^\star$ satisfies
\begin{equation}
\label{eq:w(U*)}
w(U^\star)<cw(G).
\end{equation}
We apply the same argument as for the case $U\neq\emptyset$, but with $G^\star$ and $U^\star$ in place of $G$ and $U$ and using \eqref{eq:w(U*)} instead of \eqref{eq:w(U)}.
This gives us two cycles $C_0$ and $C_1$ in $G^\star$ with vertex sets $S_0$ and $S_1$, respectively, which satisfy \eqref{eq:S-def-1} and \eqref{eq:S-def-2}.
Moreover, each $C_k$ can contain only one edge from $E(G^\star)\setminus E(G)$, namely, the one entering $v$.
In $G$, that edge ends at a vertex from $U^\star$, which shows that each $S_k$ is connected in $G$.
\end{proof}

Lemma \ref{lem:struct-hamil1} gives a lower bound on $\balance(H\setminus S)$.
The next lemma will allow us to replace it by a lower bound on $\balance(G\setminus S)$.

\begin{lemma}
\label{lem:struct-hamil2}
There is\/ $c\in(0,1]$ such that every weighted graph\/ $G$ containing a Hamiltonian cycle\/ $H$ has the property that every connected set\/ $A\subset V(G)$ satisfies at least one of the following conditions:
\begin{enumerate}
\item\label{item:struct-hamil2-1} There is a connected set\/ $S\subset V(G)$ such that\/ $A\subset S$ and
\begin{equation*}
\balance(G\setminus S)\geq c\balance(H\setminus A).
\end{equation*}
\item\label{item:struct-hamil2-2} There is a connected set\/ $S\subset V(G)$ such that\/ $A\subset S$ and
\begin{equation*}
w(N(S))\geq c\balance(H\setminus A).
\end{equation*}
\end{enumerate}
\end{lemma}

\begin{proof}
We show that it is enough to set
\begin{equation}
\label{eq:struct-hamil2-c}
c=1/5.
\end{equation}

\begin{figure}[t]
\vspace*{-2.5ex}
\centering
\begin{tikzpicture}[scale=0.8]
  \tikzstyle{every node}=[circle,draw,minimum size=5pt,inner sep=0pt]
  \tikzstyle{every label}=[rectangle,draw=none,label distance=4pt]
  \foreach\i in {0,...,6} \node (a\i) at (\i,0) {};
  \node[fill,label=below:$u$] (a7) at (7,0) {};
  \foreach\i in {8,...,15} \node (a\i) at (\i,0) {};
  \node[fill,label=below:$v$] (a16) at (16,0) {};
  \path (a0) edge (a1) edge[bend left=30] (a2) edge[bend left=30] (a16);
  \path (a1) edge (a2) edge[bend left=35] (a5);
  \path (a2) edge (a3);
  \path (a3) edge (a4) edge[bend left=35] (a7);
  \path (a4) edge (a5);
  \path (a5) edge (a6) edge[bend left=35] (a8) edge[bend left=40] (a11);
  \path (a6) edge (a7) edge[bend left=30] (a11);
  \path (a7) edge (a8) edge[bend left=30] (a16);
  \path (a8) edge (a9) edge[bend left=30] (a10);
  \path (a9) edge (a10) edge[bend left=35] (a13);
  \path (a10) edge (a11);
  \path (a11) edge (a12);
  \path (a12) edge (a13) edge[bend left=30] (a14);
  \path (a13) edge (a14);
  \path (a14) edge (a15);
  \path (a15) edge (a16);
  \tikzstyle{every node}=[below,inner sep=3pt]
  \draw (-0.3,-0.2)--(-0.3,-0.4)--node {$B_1$}(2.3,-0.4)--(2.3,-0.2);
  \draw (2.7,-0.2)--(2.7,-0.4)--node {$B_2$}(4.3,-0.4)--(4.3,-0.2);
  \draw (4.7,-0.2)--(4.7,-0.4)--node {$B_3$}(5.3,-0.4)--(5.3,-0.2);
  \draw (5.7,-0.2)--(5.7,-0.4)--node {$B_4$}(6.3,-0.4)--(6.3,-0.2);
  \draw (7.7,-0.2)--(7.7,-0.4)--node {$B_5$}(10.3,-0.4)--(10.3,-0.2);
  \draw (10.7,-0.2)--(10.7,-0.4)--node {$B_6$}(12.3,-0.4)--(12.3,-0.2);
  \draw (12.7,-0.2)--(12.7,-0.4)--node {$B_7$}(13.3,-0.4)--(13.3,-0.2);
  \draw (13.7,-0.2)--(13.7,-0.4)--node {$B_8$}(14.3,-0.4)--(14.3,-0.2);
  \draw (14.7,-0.2)--(14.7,-0.4)--node {$B_9$}(15.3,-0.4)--(15.3,-0.2);
\end{tikzpicture}
\caption{Illustration for the proof of Lemma \ref{lem:struct-hamil2}: $A=\{u,v\}$, $\calB_0=\{B_1,B_3,B_4,B_6,B_8,B_9\}$, $\calB_1=\{B_2,B_5,B_7\}$}
\label{fig:struct-hamil2}
\end{figure}
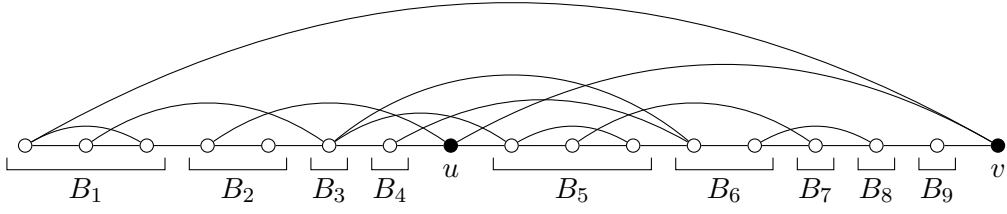

Let $G$ be a weighted graph containing a Hamiltonian cycle $H$, and let $A$ be a connected subset of $V(G)$.
If $H\setminus A$ is connected, then $\balance(H\setminus A)=0$ and the conclusion holds trivially.
Thus assume that $H\setminus A$ has at least two components.
Let $P$ be an oriented Hamiltonian path in $G$ obtained by orienting $H$ in any of the two directions and then removing the edge going out of an arbitrarily chosen vertex of $A$.
Thus the last vertex of $P$ belongs to $A$.
We partition the set $V(G)\setminus A$ into \emph{blocks} $B_1,\ldots,B_n$, which are intervals in the order $<_P$.
We construct them one by one in the order of their indices as follows.
Let $B_i$ be the interval $[f_i,u_i)_P$ such that
\begin{itemize}
\item $f_i$ is the least vertex in $<_P$ with $f_i\notin A\cup B_1\cup\cdots\cup B_{i-1}$,
\item $u_i$ is the least vertex in $<_P$ with $f_i<_Pu_i$ and $u_i\in N[A\cup B_1\cup\cdots\cup B_{i-1}]$.
\end{itemize}
Notice that if $u_i\notin A$, then $u_i=f_{i+1}$.
Now, we partition the family $\{B_1,\ldots,B_n\}$ of all blocks into two subfamilies $\calB_0$ and $\calB_1$ as follows.
We process the blocks in the order of their indices.
We put $B_i$ into $\calB_0$ if $u_i\in N[A]$ or $u_i$ is adjacent to at least one of $B_1,\ldots,B_{i-1}$ that has been already put into $\calB_1$.
Otherwise, we put $B_i$ into $\calB_1$.
See Figure \ref{fig:struct-hamil2}.
It follows from the presented construction that
\begin{enumeratea}
\item\label{item:block-a} $f_i\in N(A)$ or $f_i$ is adjacent to at least one block of $\{B_1,\ldots,B_{i-1}\}\cap\calB_0$ and at least one block of $\{B_1,\ldots,B_{i-1}\}\cap\calB_1$;
\item\label{item:block-b} no vertex in $B_i\setminus\{f_i\}$ is adjacent to $A\cup B_1\cup\cdots\cup B_{i-1}$.
\end{enumeratea}
For $k\in\{0,1\}$, define
\begin{equation*}
A'_k=A\cup\tbigcup\calB_{1-k},\qquad F_k=\{f_i\colon B_i\in\calB_k\},\qquad A''_k=A'_k\cup F_k.
\end{equation*}
It follows from \ref{item:block-a} that $G[A'_k]$ is connected---every block $B_i\in\calB_{1-k}$ is connected to $A$ directly or through blocks from $\{B_1,\ldots,B_{i-1}\}\cap\calB_{1-k}$.
It also follows from \ref{item:block-a} that $F_k\subset N(A'_k)$.
Therefore, if
\begin{equation*}
w(F_k)\geq c\balance(H\setminus A),
\end{equation*}
then 
\begin{equation*}
w(N(A'_k))\geq c\balance(H\setminus A)
\end{equation*}
and hence \ref{item:struct-hamil2-2} holds for $S=A'_k$.
Thus assume
\begin{equation}
\label{eq:w(F_k)}
w(F_k)<c\balance(H\setminus A)\quad\text{for }k\in\{0,1\}.
\end{equation}
Since $F_k\subset N(A'_k)$ and $G[A'_k]$ is connected, $G[A''_k]$ is connected too.
If
\begin{equation*}
\balance(G\setminus A''_k)\geq c\balance(H\setminus A),
\end{equation*}
then \ref{item:struct-hamil2-1} holds for $S=A''_k$.
Thus assume
\begin{equation}
\label{eq:comps}
\balance(G\setminus A''_k)<c\balance(H\setminus A)\quad\text{for }k\in\{0,1\}.
\end{equation}
By \ref{item:block-b}, the components of $G\setminus A''_k$ are precisely the subgraphs $G[B_i\setminus\{f_i\}]$ for $B_i\in\calB_k$.
Let $C_k$ be a maximum weight set of the form $B_i\setminus\{f_i\}$ with $B_i\in\calB_k$.
Since $\{F_0,V(G)\setminus A''_0,F_1,V(G)\setminus A''_1\}$ is a partition of $V(G)\setminus A$, we have
\begin{gather*}
\begin{aligned}
&w(F_0)+\balance(G\setminus A''_0)+w(C_0)+w(F_1)+\balance(G\setminus A''_1)+w(C_1)\\
&\quad=w(F_0)+w(G\setminus A''_0)+w(F_1)+w(G\setminus A''_1)\\
&\quad=w(G\setminus A)=\balance(H\setminus A)+\gamma,
\end{aligned}
\end{gather*}
where
\begin{equation*}
\gamma=\max_{C\in\calC(H\setminus A)}w(C).
\end{equation*}
The above together with \eqref{eq:w(F_k)} and \eqref{eq:comps} implies
\begin{gather}
\label{eq:w(C_0)+w(C_1)}
\begin{alignedat}{2}
w(C_0)+w(C_1)&>(1-4c)\balance(H\setminus A)+\gamma\\
&=c\balance(H\setminus A)+\gamma&&\text{by \eqref{eq:struct-hamil2-c}}.
\end{alignedat}
\end{gather}
Since each $C_k$ is contained in one component of $H\setminus A$, we have
\begin{equation}
\label{eq:w(C_k)}
w(C_k)\leq\gamma\quad\text{for }k\in\{0,1\},
\end{equation}
If $C_0$ and $C_1$ are contained in the same component of $H\setminus A$, then we have
\begin{equation*}
w(C_0)+w(C_1)\leq\gamma,
\end{equation*}
which contradicts \eqref{eq:w(C_0)+w(C_1)}.
Thus $C_0$ and $C_1$ are contained in distinct components of $H\setminus A$.
It follows that $H\setminus(C_0\cup C_1)$ consists of two components each containing a vertex from $A$.
This implies that $G\setminus(C_0\cup C_1)$ is connected, as it contains the whole connected set $A$.
Moreover, $G[C_0]$ and $G[C_1]$ are precisely the components of $G[C_0\cup C_1]$, and thus we have
\begin{gather*}
\begin{alignedat}{2}
\balance(G[C_0\cup C_1])&=w(C_0)+w(C_1)-\max\{w(C_0),w(C_1)\}\\
&\geq w(C_0)+w(C_1)-\gamma&&\text{by \eqref{eq:w(C_k)}}\\
&>c\balance(H\setminus A)&&\text{by \eqref{eq:w(C_0)+w(C_1)}}.
\end{alignedat}
\end{gather*}
This shows that \ref{item:struct-hamil2-1} is satisfied for $S=V(G)\setminus(C_0\cup C_1)$.
\end{proof}

\begin{corollary}
\label{cor:struct-hamil}
There is\/ $c\in(0,1]$ such that every weighted graph\/ $G$ containing a Hamiltonian cycle satisfies at least one of the following conditions:
\begin{enumerate}
\item\label{item:struct-hamil-1} There is a connected set\/ $S\subset V(G)$ such that
\begin{equation*}
\balance(G\setminus S)\geq cw(G).
\end{equation*}
\item\label{item:struct-hamil-2} There is a connected set\/ $S\subset V(G)$ such that
\begin{equation*}
w(N(S))\geq cw(G).
\end{equation*}
\item\label{item:struct-hamil-3} There is a set\/ $S\subset V(G)$ such that\/ $\link{G}{S}$ is a cycle and
\begin{equation*}
w(S)\geq cw(G).
\end{equation*}
\end{enumerate}
\end{corollary}

\begin{proof}
We show that it is enough to set $c=c'c''$, where $c'$ and $c''$ are the constants claimed by Lemmas \ref{lem:struct-hamil1} and \ref{lem:struct-hamil2}, respectively.

Let $G$ be a weighted graph containing a Hamiltonian cycle $H$.
By Lemma \ref{lem:struct-hamil1}, at least one of the following holds:
\begin{enumerate'}
\item\label{item:struct-hamil-1'} There is a connected set $A\subset V(G)$ such that
\begin{equation*}
\balance(H\setminus A)\geq c'w(G).
\end{equation*}
\item\label{item:struct-hamil-2'} There is a set $S\subset V(G)$ such that $\link{G}{S}$ is a cycle and
\begin{equation*}
w(S)\geq c'w(G).
\end{equation*}
\end{enumerate'}
If \ref{item:struct-hamil-2'} holds, then it directly implies \ref{item:struct-hamil-3}.
Thus assume \ref{item:struct-hamil-1'}.
By Lemma \ref{lem:struct-hamil2}, at least one of the following holds:
\begin{enumerate''}
\item There is a connected set $S\subset V(G)$ such that $A\subset S$ and
\begin{equation*}
\balance(G\setminus S)\geq c''\balance(H\setminus A)\geq cw(G).
\end{equation*}
\item There is a connected set $S\subset V(G)$ such that $A\subset S$ and
\begin{equation*}
w(N(S))\geq c''\balance(H\setminus A)\geq cw(G).
\end{equation*}
\end{enumerate''}
Hence \ref{item:struct-hamil-1} or \ref{item:struct-hamil-2} follows, respectively.
\end{proof}

\subsection{General graphs}

The following lemma, first proved by Kao \cite{Kao88}, allows us to reduce the problem for general graphs to the problem for Hamiltonian graphs.
We include a proof for the reader's convenience.

\begin{lemma}[Kao {\cite[Theorem 3]{Kao88}}]
\label{lem:struct-cycle}
Every weighted connected graph\/ $G$ contains a cycle\/ $H$ such that every component of\/ $G\setminus V(H)$ has weight at most\/ $\tfrac{1}{2}w(G)$.
\end{lemma}

\begin{proof}
Let $G$ be a weighted connected graph.
A tree $T$ is a \emph{spanning tree} of $G$ if $V(T)=V(G)$ and $E(T)\subset E(G)$.
For a spanning tree $T$ of $G$ and an edge $uv$ of $G$, let $T_{uv}$ denote the set of vertices of the unique path in $T$ connecting $u$ and $v$.
Choose an arbitrary vertex $r$ of $G$.
A standard depth-first search procedure started from $r$ constructs a spanning tree $T$ of $G$ with the following property:
\begin{equation}
\label{eq:dfs}
\text{for every edge $xy$ of $G$, either $x\in T_{ry}$ or $y\in T_{rx}$.}
\end{equation}

We find a vertex $v$ in $T$ such that every component of $T\setminus\{v\}$ has weight at most $\tfrac{1}{2}w(T)$.
This can be done as follows.
Orient every edge $uw$ of $T$ as $u\to w$ so that after removing the edge $uw$ from $T$, the total weight of the component containing $u$ is at most the total weight of the component containing $w$ (picking the orientation arbitrarily if the two weights are equal).
The oriented tree obtained this way must have a sink (a vertex of out-degree $0$).
If $v$ is a sink, then every component of $T\setminus\{v\}$ has weight at most $\tfrac{1}{2}w(T)$, otherwise the edge connecting $v$ with the component with weight greater than $\tfrac{1}{2}w(T)$ would have been oriented outwards $v$.

If $v=r$, then no edge of $G$ connects two distinct components of $T\setminus\{v\}$, as such an edge would contradict \eqref{eq:dfs}.
In this case, the conclusion follows by choosing $H=\{v\}$.
Now, suppose $v\neq r$.
Let $C$ be the component of $T\setminus\{v\}$ containing $r$, and let $C^\star$ be the union of all other components of $T\setminus\{v\}$.
By \eqref{eq:dfs}, all edges of $G$ connecting two distinct components of $T\setminus\{v\}$ go between $C$ and $C^\star$.
Let $xy$ be an edge of $G$ connecting $x\in V(C)$ and $y\in V(C^\star)$ and minimizing the distance between $r$ and $x$ in $T$.
It follows that $x\in T_{rv}$, otherwise $xy$ would contradict \eqref{eq:dfs}.
Every other edge $x'y'\in E(G)$ connecting $x'\in V(C)$ and $y'\in V(C^\star)$ also satisfies $x'\in T_{rv}$, and therefore, by the choice of $x$, it satisfies $x'\in T_{xv}$.
This shows that the vertex set of every component of $G\setminus T_{xv}$ is entirely contained in one component of $T\setminus\{v\}$.
Let $H$ be the cycle formed by the edge $xy$ and the unique path in $T$ between $x$ and $y$.
Since $T_{xv}\subset V(H)$, the vertex set of every component of $G\setminus V(H)$ is entirely contained in one component of $T\setminus\{v\}$.
Therefore, since every component of $T\setminus\{v\}$ has weight at most $\tfrac{1}{2}w(T)$, every component of $G\setminus V(H)$ has weight at most $\tfrac{1}{2}w(G)$.
\end{proof}

\begin{corollary}
\label{cor:struct-full}
There is a constant $c\in(0,1]$ such that every weighted connected graph $G$ satisfies at least one of the following conditions:
\begin{enumerate}
\item\label{item:struct-full-1} There is a connected set $S\subset V(G)$ such that
\begin{equation*}
\balance(G\setminus S)\geq cw(G).
\end{equation*}
\item\label{item:struct-full-2} There is a connected set $S\subset V(G)$ such that
\begin{equation*}
w(N(S))\geq cw(G).
\end{equation*}
\item\label{item:struct-full-3} There is a set $S\subset V(G)$ such that $\link{G}{S}$ is a cycle and
\begin{equation*}
w(S)\geq cw(G).
\end{equation*}
\end{enumerate}
\end{corollary}

\begin{proof}
Let $c'\in(0,1]$ be the constant claimed by Corollary \ref{cor:struct-hamil}.
Set
\begin{equation*}
c=\frac{c'}{2(1+c')}.
\end{equation*}
It follows that
\begin{equation}
\label{eq:struct-full-c}
c'(\tfrac{1}{2}-c)=c.
\end{equation}
We show that this is enough for the conclusion of the lemma.

Let $G$ be a weighted connected graph.
By Lemma \ref{lem:struct-cycle}, there is a cycle $H$ in $G$ such that every component of $G\setminus V(H)$ has weight at most $\tfrac{1}{2}w(G)$.
If $w(H)\leq(\tfrac{1}{2}-c)w(G)$, then the conclusion \ref{item:struct-full-1} with $S=V(H)$ follows:
\begin{equation*}
\balance(G\setminus S)\geq w(G\setminus S)-\tfrac{1}{2}w(G)=\tfrac{1}{2}w(G)-w(H)\geq cw(G).
\end{equation*}
Thus assume $w(H)>(\tfrac{1}{2}-c)w(G)$.
Let $G'=\link{G}{V(H)}$.
By \eqref{eq:struct-full-c}, we have
\begin{equation*}
c'w(G')=c'w(H)>c'(\tfrac{1}{2}-c)w(G)=cw(G).
\end{equation*}
By Corollary \ref{cor:struct-hamil} and by the above, at least one of the following holds:
\begin{enumerate'}
\item\label{item:struct-full-1'} There is a connected set $S'\subset V(G')$ such that
\begin{equation*}
\balance(G'\setminus S')\geq c'w(G')>cw(G).
\end{equation*}
\item\label{item:struct-full-2'} There is a connected set $S'\subset V(G')$ such that
\begin{equation*}
w(N_{G'}(S'))\geq c'w(G')>cw(G).
\end{equation*}
\item\label{item:struct-full-3'} There is a set $S'\subset V(G')$ such that $\link{G'}{S'}$ is a cycle and
\begin{equation*}
w(S')\geq c'w(G')>cw(G).
\end{equation*}
\end{enumerate'}
We show that each of the statements \ref{item:struct-full-1'}--\ref{item:struct-full-3'} above implies the corresponding statement \ref{item:struct-full-1}--\ref{item:struct-full-3} in the conclusion of the lemma.

Suppose that \ref{item:struct-full-1'} holds.
Let $S$ be the set of all vertices of $G$ reachable in $G$ by a path starting in $S'$ and containing no other vertex of $G'$.
Clearly, $S\cap V(G')=S'$.
If $uv$ is an edge of $G'[S']$, then the whole path from $u$ to $v$ in $G$ witnessing the edge $uv$ in $G'$ belongs to $S$.
Therefore, since $G'[S']$ is connected, $G[S]$ is connected too.
Moreover, any two vertices of $G'\setminus S'$ are connected by a path in $G'\setminus S'$ if and only if they are connected by a path in $G\setminus S$.
Consequently, if $C$ is a component of $G\setminus S$, then $C\cap V(G')$ is a component of $G'\setminus S'$ and hence
\begin{gather*}
\begin{aligned}
w(C)&=w(C\cap V(G'))+w(C\setminus V(G'))\\
&\leq\max_{C'\in\calC(G'\setminus S')}w(C')+w\bigl(G\setminus(S\cup V(G'))\bigr).
\end{aligned}
\end{gather*}
We conclude that \ref{item:struct-full-1} holds:
\begin{gather*}
\begin{aligned}
\balance(G\setminus S)&=w(G\setminus S)-\max_{C\in\calC(G\setminus S)}w(C)\\
&\geq w(G'\setminus S')+w\bigl(G\setminus(S\cup V(G'))\bigr)\\
&\qquad-\max_{C'\in\calC(G'\setminus S')}w(C')-w\bigl(G\setminus(S\cup V(G'))\bigr)\\
&=\balance(G'\setminus S')>cw(G).
\end{aligned}
\end{gather*}

Now, suppose that \ref{item:struct-full-2'} holds.
Again, let $S$ be the set of all vertices in $G$ reachable in $G$ by a path starting in $S'$ and containing no other vertex of $G'$.
As before, $S\cap V(G')=S'$ and $G[S]$ is connected.
Moreover, if $uv$ is an edge in $G'$ such that $u\in S'$ and $v\in N_{G'}(S')$, then $v\in N_G(S)$, as the entire path from $u$ to $v$ in $G$ witnessing the edge $uv$ in $G'$ except the vertex $v$ is included in $S$.
Therefore, $N_{G'}(S')\subset N_G(S)$ and hence \ref{item:struct-full-2} follows:
\begin{equation*}
w(N_G(S))\geq w(N_{G'}(S'))>cw(G).
\end{equation*}

Finally, suppose that \ref{item:struct-full-3'} holds.
Let $S=S'$.
We have
\begin{equation*}
w(S)=w(S')>cw(G).
\end{equation*}
Moreover, $\link{G}{S}=\link{G'}{S'}$ and hence \ref{item:struct-full-3} follows.
\end{proof}

\subsection{Graphs with a forbidden subdivision}

Recall that a graph $G$ contains a subdivision of a graph $H$ if $G$ has a subgraph $F$ that arises from $H$ by replacing every edge $uv\in E(H)$ by a path $F_{uv}$ between $u$ and $v$ so that the paths $F_{uv}$ are internally disjoint from $V(H)$ and from each other.
Hence the vertices of $H$ maintain their identity in the subdivision and, in particular, $V(H)\subset V(G)$.
This subtlety is important for the following lemma.

\begin{lemma}
\label{lem:struct-indsubdiv}
For any\/ $n\in\setN^+$ and\/ $m\in\{0,\ldots,\tbinom{n}{2}\}$, there is\/ $c_{n,m}\in(0,1]$ such that every weighted connected graph\/ $G$ and every connected non-empty proper subset\/ $A$ of\/ $V(G)$ satisfy at least one of the following conditions:
\begin{enumerate}
\item\label{item:struct-indsubdiv-1} There is a connected set\/ $S\subset V(G)$ such that\/ $A\subset S$ and
\begin{equation*}
w(N(A)\setminus S)-\max_{C\in\calC(G\setminus S)}w(C\cap N(A))\geq c_{n,m}w(N(A)).
\end{equation*}
\item\label{item:struct-indsubdiv-2} There is a vertex\/ $v\in N(A)$ such that
\begin{equation*}
w(v)\geq c_{n,m}w(N(A)).
\end{equation*}
\item\label{item:struct-indsubdiv-3} There is a graph\/ $H$ with\/ $n$ vertices and\/ $m$ edges such that\/ $V(H)\subset N(A)$ and\/ $G\setminus A$ contains a subdivision of\/ $H$.
\end{enumerate}
\end{lemma}

\begin{proof}
Fix $n\in\setN^+$.
If $n=1$, then \ref{item:struct-indsubdiv-3} holds trivially, so assume $n\geq 2$.
The proof goes by induction on $m$.
If $m=0$, then it is enough to set
\begin{equation*}
c_{n,0}=\frac{1}{n-1}.
\end{equation*}
Indeed, let $G$ be a weighted connected graph and $A$ be a connected non-empty proper subset of $V(G)$.
If $\size{N(A)}\leq n-1$, then the heaviest vertex in $N(A)$ satisfies \ref{item:struct-indsubdiv-2}.
If $\size{N(A)}\geq n$, then $G\setminus A$ contains a subgraph $H$ with $n$ vertices and no edges.

Thus assume $1\leq m\leq\tbinom{n}{2}$.
We show that it is enough to set
\begin{gather}
\label{eq:struct-ind-beta}
\beta_{n,m}=\frac{c_{n,m-1}}{2(1+c_{n,m-1})},\\
\label{eq:struct-ind-c}
c_{n,m}=\frac{\beta_{n,m}}{1+\beta_{n,m}}.
\end{gather}

Let $G$ be a weighted connected graph and $A$ be a connected non-empty proper subset of $V(G)$.
If every component $C$ of $G\setminus A$ satisfies
\begin{equation*}
w(C\cap N(A))\leq(1-c_{n,m})w(N(A)),
\end{equation*}
then \ref{item:struct-indsubdiv-1} holds for $S=A$.
Thus assume there is a component $C$ of $G\setminus A$ with
\begin{equation*}
w(C\cap N(A))>(1-c_{n,m})w(N(A)).
\end{equation*}
Define $B=V(C)\cap N(A)$.
It follows that
\begin{gather}
\label{eq:singlecomp}
\begin{aligned}
w(B)&>(1-c_{n,m})w(N(A))\\
&=\frac{1}{1+\beta_{n,m}}w(N(A))\qquad\text{by \eqref{eq:struct-ind-c}}.
\end{aligned}
\end{gather}
Choose any $v_0\in B$.
For $i\in\setN$, let $B_i$ be the set of vertices at distance $i$ from $v_0$ in $\link{C}{B}$.
Let $k$ be the greatest index for which $B_k\neq\emptyset$.
Clearly, the sets $B_0,\ldots,B_k$ form a partition of $B$.
Moreover, $B_i$ and $B_j$ are not adjacent in $\link{C}{B}$ whenever $\size{i-j}\geq 2$.
Choose $j\in\{0,\ldots,k\}$ so that
\begin{gather}
\label{eq:levels}
\begin{aligned}
w(B_0\cup\cdots\cup B_{j-1})&\leq\tfrac{1}{2}w(B),\\
w(B_{j+1}\cup\cdots\cup B_k)&\leq\tfrac{1}{2}w(B).
\end{aligned}
\end{gather}
Suppose
\begin{equation*}
w(B_j)\leq(\tfrac{1}{2}-\beta_{n,m})w(B).
\end{equation*}
It follows from the above and \eqref{eq:levels} that
\begin{gather}
\label{eq:levels2}
\begin{alignedat}{2}
w(B_0\cup\cdots\cup B_{j-1})&=w(B)-w(B_j\cup B_{j+1}\cup\cdots\cup B_k)&&\geq\beta_{n,m}w(B),\\
w(B_{j+1}\cup\cdots\cup B_k)&=w(B)-w(B_0\cup\cdots\cup B_{j-1}\cup B_j)&&\geq\beta_{n,m}w(B).
\end{alignedat}
\end{gather}
Since there is no path in $G$ connecting $B_0\cup\cdots\cup B_{j-1}$ and $B_{j+1}\cup\cdots\cup B_k$ that avoids $A\cup B_j$, the vertex set of every component of $G\setminus(A\cup B_j)$ is disjoint from $B_0\cup\cdots\cup B_{j-1}$ or $B_{j+1}\cup\cdots\cup B_k$.
Thus \ref{item:struct-indsubdiv-1} follows for $S=A\cup B_j$:
\begin{gather*}
\begin{alignedat}{2}
w(N(A)\setminus S)-\smash[b]{\max_{C\in\calC(G\setminus S)}w(C\cap N(A))}&\geq\beta_{n,m}w(B)&\quad&\text{by \eqref{eq:levels2}}\\
&>c_{n,m}w(N(A))&\quad&\text{by \eqref{eq:singlecomp} and \eqref{eq:struct-ind-c}}.
\end{alignedat}
\end{gather*}
It remains to consider the case
\begin{equation}
\label{eq:level}
w(B_j)>(\tfrac{1}{2}-\beta_{n,m})w(B).
\end{equation}

Every vertex in $B_j$ is reachable in $C$ from $v_0$ by a path avoiding all other vertices from $B_j$.
Let $A'$ be the set of vertices of $G$ reachable in $G$ from $A$ by a path entirely disjoint from $B_j$.
It follows that $A\subset A'\subset V(G)\setminus B_j$ and
\begin{equation}
\label{eq:N(A')}
N(A')=B_j=N(A)\setminus A'.
\end{equation}
Since $G[A]$ is connected, $G[A']$ is connected too.
By the induction hypothesis applied to $G$ and the proper subset $A'$ of $V(G)$, at least one of the following holds:
\begin{enumerate'}
\item\label{item:struct-indsubdiv-1'} There is a connected set $S\subset V(G)$ such that $A'\subset S$ and
\begin{equation*}
w(B_j\setminus S)-\max_{C\in\calC(G\setminus S)}w(C\cap B_j)\geq c_{n,m-1}w(B_j).
\end{equation*}
\item\label{item:struct-indsubdiv-2'} There is a vertex $v\in B_j$ such that
\begin{equation*}
w(v)\geq c_{n,m-1}w(B_j).
\end{equation*}
\item\label{item:struct-indsubdiv-3'} There is a graph $H'$ with $n$ vertices and $m-1$ edges such that $V(H')\subset B_j$ and $G\setminus A'$ contains a subdivision of $H'$.
\end{enumerate'}
Moreover, we have
\begin{gather}
\label{eq:subdiv-aux}
\begin{alignedat}{2}
c_{n,m-1}w(B_j)&>\beta_{n,m}w(B)&\quad&\text{by \eqref{eq:level} and \eqref{eq:struct-ind-beta}}\\
&>c_{n,m}w(N(A))&\quad&\text{by \eqref{eq:singlecomp} and \eqref{eq:struct-ind-c}}.
\end{alignedat}
\end{gather}
If \ref{item:struct-indsubdiv-1'} holds, then by \eqref{eq:N(A')} we have $B_j\setminus S=N(A)\setminus S$ and $V(C)\cap B_j=V(C)\cap N(A)$ for every $C\in\calC(G\setminus S)$.
This and \eqref{eq:subdiv-aux} imply \ref{item:struct-indsubdiv-1} for the same set $S$.
If \ref{item:struct-indsubdiv-2'} holds, then by \eqref{eq:subdiv-aux} we have \ref{item:struct-indsubdiv-2} for the same vertex $v$.
So suppose that \ref{item:struct-indsubdiv-3'} holds.

Let $u$ and $v$ be any two vertices of $H'$ such that $uv\notin E(H')$.
To prove \ref{item:struct-indsubdiv-3}, we show that $G\setminus A$ contains a subdivision of the graph $H$ with $V(H)=V(H')$ and $E(H)=E(H')\cup\{uv\}$.
Let $F'$ be a subdivision of $H'$ in $G\setminus A'$ claimed by \ref{item:struct-indsubdiv-3'}.
Since $u,v\in B_j$, the vertices $u$ and $v$ are reachable in $G\setminus A$ from $v_0$ by paths $P_u$ and $P_v$, respectively, avoiding all other vertices from $B_j$.
Let $P$ be the path connecting $u$ and $v$ in $P_u\cup P_v$.
Since $v_0\in N(A)$, it follows from the definition of $A'$ that $V(P_u)\setminus\{u\}\subset A'$ and $V(P_v)\setminus\{v\}\subset A'$, and thus $V(P)\setminus\{u,v\}\subset A'$.
In particular, $P$ is internally disjoint from $F'$.
This shows that $F'\cup P$ is a subdivision of $H$ in $G\setminus A$.
\end{proof}

\begin{corollary}[Theorem \ref{thm:intro-struct} rephrased]
\label{cor:struct-subdiv}
For every\/ $n\in\setN^+$, there is\/ $c_n\in(0,1]$ such that every weighted connected graph\/ $G$ containing no subdivision of\/ $K_n$ satisfies at least one of the following conditions:
\begin{enumerate}
\item\label{item:struct-subdiv-1} There is a connected set\/ $S\subset V(G)$ such that
\begin{equation*}
\balance(G\setminus S)\geq c_nw(G).
\end{equation*}
\item\label{item:struct-subdiv-2} There is a set\/ $S\subset V(G)$ such that\/ $\link{G}{S}$ is a cycle and
\begin{equation*}
w(S)\geq c_nw(G).
\end{equation*}
\end{enumerate}
\end{corollary}

\begin{proof}
Fix $n\in\setN^+$.
We show that it is enough to set $c_n=c'c''_n$, where $c'$ is the constant claimed by Corollary \ref{cor:struct-full}, and $c''_n$ is the constant claimed by Lemma \ref{lem:struct-indsubdiv} for $m=\tbinom{n}{2}$.

Let $G$ be a weighted connected graph containing no subdivision of $K_n$.
It follows from Corollary \ref{cor:struct-full} that \ref{item:struct-subdiv-1} or \ref{item:struct-subdiv-2} holds or there is a connected set $A\subset V(G)$ such that
\begin{equation*}
w(N(A))\geq c'w(G).
\end{equation*}
In the latter case, by Lemma \ref{lem:struct-indsubdiv}, at least one of the following holds:
\begin{enumerate'}
\item\label{item:struct-subdiv-1'} There is a connected set $S\subset V(G)$ such that $A\subset S$ and
\begin{equation*}
w(N(A)\setminus S)-\max_{C\in\calC(G\setminus S)}w(C\cap N(A))\geq c''_nw(N(A))\geq c_nw(G).
\end{equation*}
\item\label{item:struct-subdiv-2'} There is a vertex $v\in N(A)$ such that
\begin{equation*}
w(v)\geq c''_nw(N(A))\geq c_nw(G).
\end{equation*}
\end{enumerate'}
The third case of Lemma \ref{lem:struct-indsubdiv} is excluded by the assumption that $G$ contains no subdivision of $K_n$.
If \ref{item:struct-subdiv-1'} holds, then \ref{item:struct-subdiv-1} follows for the same set $S$:
\begin{gather*}
\begin{aligned}
\balance(G\setminus S)&=w(G\setminus S)-\max_{C\in\calC(G\setminus S)}w(C)\\
&\geq w(N(A)\setminus S)-\max_{C\in\calC(G\setminus S)}w(C\cap N(A))\\
&\geq c_nw(G).
\end{aligned}
\end{gather*}
If \ref{item:struct-subdiv-2'} holds, then \ref{item:struct-subdiv-2} follows for $S=\{v\}$.
\end{proof}

Solving the recurrence formula for $c_{n,m}$ in the proof of Lemma \ref{lem:struct-indsubdiv} gives $c_{n,m}=\Theta(2^{-m}n^{-1})$.
Consequently, the formula for $c_n$ in the proof of Corollary \ref{cor:struct-subdiv} gives $c_n=\Theta(2^{-\smash{\binom{n}{2}}}n^{-1})$.

\section{Strategies}
\label{sec:strategies}

This final section is devoted to the proof of Theorem \ref{thm:intro-game}.
Namely, for a suitable constant $c_n>0$, we show that Alice can secure at least $c_nw(G)$ in the graph sharing game played on any weighted connected graph $G$ with an odd number of vertices and with no subdivision of $K_n$.
For the entire section, we assume that $G$ is a fixed weighted connected graph with vertex set $V$.
The additional conditions that $G$ has no subdivision of $K_n$ or $\size{V}$ is odd will be explicitly stated wherever they are required.

We call a set $S\subset V$ \emph{sparse} if the distance in $G$ between any two vertices in $S$ is at least $3$.
Equivalently, $S$ is sparse if the closed neighborhoods of the vertices in $S$ are pairwise disjoint.
We call $G$ \emph{sparsely weighted} if the set of vertices of $G$ with positive weight is sparse.

First, we prove that Alice has a strategy to gather at least $c_nw(G)$ if $G$ is a sparsely weighted graph with an odd number of vertices and with no subdivision of $K_n$, for a suitable constant $c_n>0$.
This strategy can be as well applied when $G$, instead of being sparsely weighted, contains a sparse set of vertices with substantial weight (at least a constant proportion of $w(G)$).
Then, to prove the theorem for any graph $G$ with an odd number of vertices and with no subdivision of $K_n$, we present complementary strategies of Alice that work when no sparse set of vertices has substantial weight.

\subsection{Strategies on sparsely weighted graphs}

The idea behind the proof of Theorem \ref{thm:intro-game} for sparsely weighted graphs is to devise a strategy for each of the two cases resulting from Corollary \ref{cor:struct-subdiv}.
The following lemma is the heart of the strategy for the case \ref{item:struct-subdiv-1}.

\begin{lemma}[{\cite[Lemma 3.1]{MiW12}}]
\label{lem:strat-comp}
Assume\/ $\size{V}$ is odd.
Consider an intermediate position in the graph sharing game on\/ $G$ at which a set\/ $T$ of vertices has been taken and Alice is to move (\/$\size{T}$ is even).
Starting from that position, Alice has a strategy to collect vertices of total weight at least\/ $\tfrac{1}{2}\balance(G\setminus T)$.
\end{lemma}

For the rest of this subsection, we assume that $G$ is sparsely weighted and $\size{V}$ is odd.
The problem with applying Corollary \ref{cor:struct-subdiv} and then Lemma \ref{lem:strat-comp} directly to $G$ is that when Alice is taking vertices from the separating set $S$ in order to reach a position at which the whole $S$ has been taken, Bob can steal some valuable vertices from $V\setminus S$.
This can be prevented if $S\cap N[v]=\emptyset$ for every vertex $v\in V\setminus S$ with positive weight.
To ensure the latter whenever we are in the case \ref{item:struct-subdiv-1} of Corollary \ref{cor:struct-subdiv}, we are going to contract the closed neighborhood $N[v]$ of every vertex $v$ with positive weight, thus obtaining a $1$-shallow minor $G^R$ of $G$, and apply Corollary \ref{cor:struct-subdiv} to $G^R$ instead of $G$.

Let $V^+$ denote the set of vertices of $G$ with positive weight.
Hence $V^+$ is sparse.
Define
\begin{equation*}
V^R=\bigl\{N[v]\colon v\in V^+\bigr\}\cup\bigl\{\{v\}\colon v\in V\setminus N[V^+]\bigr\}.
\end{equation*}
It follows that $V^R$ is a partition of $V$ into $1$-shallow subsets.
Define
\begin{equation*}
G^R=G/V^R.
\end{equation*}
Thus $G^R$ is a $1$-shallow minor of $G$.
It is weighted by the weight function $w$ inherited from $G$ as follows: for $v^R\in V^R$, we have
\begin{equation*}
w(v^R)=\begin{cases}
w(v)&\text{if }v^R=N[v]\text{ for some }v\in V^+,\\
0&\text{if }v^R=\{v\}\text{ for some }v\in V\setminus N[V^+].
\end{cases}
\end{equation*}
In particular, we have $w(G^R)=w(G)$.

\begin{lemma}
\label{lem:strat-comp-R}
For every set\/ $S^R\subset V^R$ that is connected in\/ $G^R$, Alice has a strategy in the graph sharing game on\/ $G$ to collect vertices of total weight at least\/ $\tfrac{1}{2}\balance(G^R\setminus S^R)$.
\end{lemma}

\begin{proof}
Let $S^R$ be a connected set of vertices of $G^R$.
Let $S=\tbigcup S^R\subset V$.
Clearly, $S$ is connected in $G$ and disjoint from $N[V^+\setminus S]$.
Alice starts by taking an arbitrary vertex from $S$.
Whenever Bob takes a vertex from $N(v)$ for some $v\in V^+\setminus S$, Alice answers by taking $v$.
Otherwise, unless the entire $S$ has been taken, Alice picks a next available vertex from $S$.
Now, consider Alice's first turn before which the entire $S$ has been taken.
Let $T$ be the set of vertices taken thus far.
Thus $S\subset T$.
Since all vertices in $T\setminus S$ with positive weight have been taken by Alice, she has already gathered at least $w(T\setminus S)$.
Alice continues the game with her strategy claimed by Lemma \ref{lem:strat-comp}.
This way, she is still going to take at least $\tfrac{1}{2}\balance(G\setminus T)$.
Therefore, her total outcome on $G$ is at least
\begin{gather*}
\begin{alignedat}[b]{1}
w(T\setminus S)+\tfrac{1}{2}\balance(G\setminus T)&=w(T\setminus S)+\tfrac{1}{2}w(G\setminus T)-\tfrac{1}{2}\max_{C\in\calC(G\setminus T)}w(C)\\
&\geq\tfrac{1}{2}w(G\setminus S)-\tfrac{1}{2}\max_{C\in\calC(G\setminus S)}w(C)\\
&=\tfrac{1}{2}w(G^R\setminus S^R)-\tfrac{1}{2}\max_{C^R\in\calC(G^R\setminus S^R)}w(C^R)\\
&=\tfrac{1}{2}\balance(G^R\setminus S^R).
\end{alignedat}\qedhere
\end{gather*}
\end{proof}

\begin{lemma}
\label{lem:strat-cycle-R}
For every set\/ $S^R\subset V^R$ such that\/ $\link{G^R}{S^R}$ is a cycle, Alice has a strategy in the graph sharing game on\/ $G$ to collect vertices of total weight at least\/ $\tfrac{1}{6}w(S^R)$.
\end{lemma}

\begin{proof}
Let $S^R$ be a subset of $V^R$ such that $\link{G^R}{S^R}$ is a cycle.
We can assume without loss of generality that $S^R$ consists only of vertices of $G^R$ of the form $N[v]$ with $v\in V^+$.
Indeed, all vertices in $S^R$ of the form $\{v\}$ with $v\in V\setminus N[V^+]$ can be removed from $S^R$ without changing the weight of $S^R$ or violating the condition that $\link{G^R}{S^R}$ is a cycle.
Let
\begin{equation*}
S=\bigl\{v\in V^+\colon N[v]\in S^R\bigr\}.
\end{equation*}
It follows that
\begin{equation*}
S^R=\{N[v]\colon v\in S\}.
\end{equation*}
Let $n=\size{S}=\size{S^R}$.
If $n\leq 6$, then Alice can take the heaviest vertex in $S$ and thus gather at least $\tfrac{1}{6}w(S^R)$ with her first move.
Thus assume $n>6$.

Enumerate the vertices in $S$ as $v_0,\ldots,v_{n-1}$ in such a way that the vertices $N[v_0],\ldots,N[v_{n-1}]$ occur in this order on the cycle $\link{G^R}{S^R}$.
Let $v_n=v_0$.
Since $\link{G^R}{S^R}$ is a cycle, the neighborhood in $G^R$ of every component of $G^R\setminus S^R$ consists of either a single vertex $N[v_i]$ or two consecutive vertices $N[v_i]$ and $N[v_{i+1}]$.
Therefore, the neighborhood in $G$ of every component of $G\setminus N[S]$ is adjacent to either one set $N(v_i)$ or two consecutive sets $N(v_i)$ and $N(v_{i+1})$.
For $0\leq i<n$, let $A_i$ denote the union of $N[v_i]$ and all~components of $G\setminus N[S]$ adjacent only to $N(v_i)$, and let $A_{i,i+1}$ denote the union of all components of $G\setminus N[S]$ adjacent to both $N(v_i)$ and $N(v_{i+1})$.
The sets $A_i$ and $A_{i,i+1}$ together form a partition of $V$.
For $0\leq i\leq j<n$, define
\begin{alignat*}{2}
A(i,j)&=&&A_{i,i+1}\cup A_{i+1}\cup A_{i+1,i+2}\cup\cdots\cup A_{j-1}\cup A_{j-1,j},\\
A[i,j)&=A_i\cup{}&&A_{i,i+1}\cup A_{i+1}\cup A_{i+1,i+2}\cup\cdots\cup A_{j-1}\cup A_{j-1,j},\\
A[i,j]&=A_i\cup{}&&A_{i,i+1}\cup A_{i+1}\cup A_{i+1,i+2}\cup\cdots\cup A_{j-1}\cup A_{j-1,j}\cup A_j.
\end{alignat*}
In particular, for $0\leq i<n$, we have $A(i,i)=A[i,i)=\emptyset$ and $A[i,i]=A_i$.
Define
\begin{align*}
C_0&=\bigl\{v_i\in S\colon\size{A[0,i)}\text{ is even}\bigr\},\\
C_1&=\bigl\{v_i\in S\colon\size{A[0,i)}\text{ is odd}\bigr\}.
\end{align*}
Thus $C_0\cup C_1=S$.
Choose $C=C_0$ or $C=C_1$ so that $w(C)\geq\tfrac{1}{2}w(S)$.
It follows from the above definitions that
\begin{equation}
\label{eq:interval-even}
\text{if $v_i,v_j\in C$, then $\size{A[i,j)}$ is even, for $0\leq i<j<n$.}
\end{equation}
Define
\begin{gather}
\label{eq:S_0,S_1}
\begin{aligned}
S_0&=\bigl\{v_i\in S\colon\size{A_i}\text{ is even}\bigr\},\\
S_1&=\bigl\{v_i\in S\colon\size{A_i}\text{ is odd}\bigr\}.
\end{aligned}
\end{gather}
Thus $S_0\cup S_1=S$.
We prove the following two claims:
\begin{enumerate}
\item\label{item:cycle-1} Alice has a strategy to secure at least $w(S_0\cap C)$.
\item\label{item:cycle-2} Alice has a strategy to secure at least $\tfrac{1}{2}w(S_1\cap C)$.
\end{enumerate}
This suffices for the conclusion of the lemma: if $w(S_0\cap C)\geq\smash[b]{\tfrac{1}{3}}w(C)$, then Alice can choose a strategy claimed by \ref{item:cycle-1}, while if $w(S_1\cap C)\geq\tfrac{2}{3}w(C)$, then she can choose a strategy claimed by \ref{item:cycle-2}.

First, we present a strategy for Alice claimed by \ref{item:cycle-1}.
She starts by taking $v_0$.
Then, she sticks to the following two rules at each her turn:
\begin{itemize}
\item Always take a vertex from $S$ if any is available.
\item Never take a vertex from $N(v_i)$ for a non-taken $v_i\in S$ unless forced to.
\end{itemize}
The first rule ensures that the vertices taken from $S$ always form an interval in the cyclic order on $S$.
Suppose that at some point of the game, Alice is forced to take a vertex from $N(v_i)$ for some non-taken vertex $v_i\in S_0\cap C$.
It follows that the set of non-taken vertices is of the form $A[\ell,r]$ for some $\ell$ and $r$ with $1\leq\ell\leq r<n$ and $v_\ell,v_r\in S_0\cap C$.
Since $v_\ell,v_r\in C$, it follows from \eqref{eq:interval-even} that $\size{A[\ell,r)}$ is even.
Since $v_r\in S_0$, it follows from \eqref{eq:S_0,S_1} that $\size{A_r}$ is even.
Hence $\size{A[\ell,r]}$ is even.
On the other hand, since $G$ has an odd number of vertices and Alice is to move, the number of non-taken vertices, which is $\size{A[\ell,r]}$, is odd.
This contradiction shows that Alice is never forced to take a vertex from $N(v_i)$ for any non-taken $v_i\in S_0\cap C$, and thus Bob never gets the opportunity to take a vertex from $S_0\cap C$.
Therefore, Alice gathers the whole $S_0\cap C$.

For Alice's strategy claimed by \ref{item:cycle-2}, fix an index $j\in\{0,\ldots,n-1\}$ so that
\begin{gather}
\label{eq:intervals}
\begin{aligned}
w\bigl(\{v_0,\ldots,v_j\}\cap S_1\cap C\bigr)&\geq\tfrac{1}{2}w(S_1\cap C),\\
w\bigl(\{v_j,\ldots,v_{n-1}\}\cap S_1\cap C\bigr)&\geq\tfrac{1}{2}w(S_1\cap C).
\end{aligned}
\end{gather}
Alice starts by taking $v_j$.
Then, at each her turn, she obeys the same two rules as before.
Again, by the first rule, the vertices taken from $S$ form an interval in the cyclic order on $S$.
If Alice is never forced to take a vertex from $N(v_i)$ for any non-taken vertex $v_i\in S_1\cap C$, then Bob never gets the opportunity to take a vertex from $S_1\cap C$, so Alice takes the whole $S_1\cap C$.
Otherwise, consider the first position in the game at which Alice is forced to take a vertex from $N(v_i)$ for some non-taken vertex $v_i\in S_1\cap C$.
This is Alice's first turn after which Bob has the opportunity to take a vertex from $S_1\cap C$.
Suppose that $v_0$ and $v_{n-1}$ have not been taken yet.
Let $\ell$ and $r$ be such that $0\leq\ell<r<n$ and $v_{\ell+1},\ldots,v_{r-1}$ are the vertices taken from $S$.
Thus $v_\ell,v_r\in S_1\cap C$, and the set of all taken vertices is equal to $A(\ell,r)$.
Since $v_\ell,v_r\in C$, it follows from \eqref{eq:interval-even} that $\size{A[\ell,r)}$ is even.
Since $v_\ell\in S_1$, it follows from \eqref{eq:S_0,S_1} that $\size{A_\ell}$ is odd.
Hence $\size{A(\ell,r)}$ is odd.
On the other hand, since $A(\ell,r)$ is the set of taken vertices and Alice is to move, $\size{A(\ell,r)}$ is even.
This contradiction shows that at least one of $v_0$ and $v_{n-1}$ have been already taken and thus all $v_0,\ldots,v_j$ or all $v_j,\ldots,v_{n-1}$ have been taken.
Since all vertices from $S_1\cap C$ taken thus far have been taken by Alice, it follows from \eqref{eq:intervals} that she has gathered at least $\tfrac{1}{2}w(S_1\cap C)$.
\end{proof}

\begin{corollary}
\label{cor:strat-sparse}
For every\/ $n\in\setN^+$, there is\/ $c_n\in(0,1]$ such that if\/ $G$ is a sparsely weighted graph with an odd number of vertices and with no subdivision of\/ $K_n$, then Alice has a strategy in the graph sharing game on\/ $G$ to collect vertices of total weight at least\/ $c_nw(G)$.
\end{corollary}

\begin{proof}
Fix $n\in\setN^+$.
By Theorem \ref{thm:shallow-subdiv}, there is $N\in\setN^+$ such that if $G$ contains no subdivision of $K_n$, then $G^R$ contains no subdivision of $K_N$.
By Corollary \ref{cor:struct-subdiv}, there is $c'_N\in(0,1]$ such that if $G^R$ contains no subdivision of $K_N$, then at least one of the following holds:
\begin{enumerate}
\item\label{item:strat-sparse-1} There is a connected set $S^R\subset V^R$ such that
\begin{equation*}
\balance(G^R\setminus S^R)\geq c'_Nw(G^R).
\end{equation*}
\item\label{item:strat-sparse-2} There is a set $S^R\subset V^R$ such that $\link{G^R}{S^R}$ is a cycle and
\begin{equation*}
w(S^R)\geq c'_Nw(G^R).
\end{equation*}
\end{enumerate}
If \ref{item:strat-sparse-1} holds, then, by Lemma \ref{lem:strat-comp-R}, Alice has a strategy in the game on $G$ to collect vertices of total weight at least
\begin{equation*}
\tfrac{1}{2}\balance(G^R\setminus S^R)\geq\tfrac{1}{2}c'_Nw(G^R).
\end{equation*}
If \ref{item:strat-sparse-2} holds, then, by Lemma \ref{lem:strat-cycle-R}, Alice has a strategy in the game on $G$ to collect vertices of total weight at least
\begin{equation*}
\tfrac{1}{6}w(S^R)\geq\tfrac{1}{6}c'_Nw(G^R).
\end{equation*}
Therefore, the conclusion of the corollary holds with $c_n=\tfrac{1}{6}c'_N$.
\end{proof}

\subsection{Greedy strategies}

A set $I\subset V$ is \emph{independent} if no two vertices in $I$ are adjacent.
We assume in this subsection that $\size{V}\geq 2$, so $I$ is never the whole $V$.
The strategy claimed by Corollary \ref{cor:strat-sparse} is enough for the proof of Theorem \ref{thm:intro-game} provided that $G$ contains a sparse set of vertices whose weight is a substantial proportion of $w(G)$.
Using Theorem \ref{thm:mader} iteratively, we can always find an independent set of weight linear in $w(G)$ in a graph $G$ with no subdivision of $K_n$, but such a heavy sparse set may not exist.
For instance, a star with weights uniformly distributed on the leaves has no subdivision of $K_3$ and no heavy sparse set.
However, this case can be easily dealt with by a greedy strategy---to always take a leaf with maximum weight.
We are going to present a family of greedy strategies and prove that they can deal with all cases of a graph $G$ to which Corollary \ref{cor:strat-sparse} cannot be applied.
These strategies are parametrized by an independent set $I$ (supposed to carry a lot of weight) and a linear ordering $\sigma$ of $V\setminus I$, and work basically as follows: take vertices from $I$ greedily or, when no vertex in $I$ is available, take the vertices from $V\setminus I$ in the order $\sigma$.

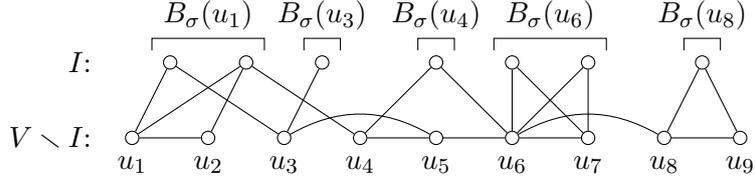
\begin{figure}[t]
\centering
\begin{tikzpicture}
  \node[left] at (0.6,0) {$V\setminus I$:};
  \node[left] at (0.6,1) {$I$:};
  \tikzstyle{every node}=[circle,draw,minimum size=5pt,inner sep=0pt]
  \tikzstyle{every label}=[rectangle,draw=none,label distance=4pt]
  \foreach\i in {1,...,9} \node[label=below:$u_\i$] (a\i) at (\i,0) {};
  \foreach\i in {1,2,3,8} \node (b\i) at (\i.5,1) {};
  \foreach\i in {5,6,7} \node (b\i) at (\i,1) {};
  \path (a1) edge (b1) edge (b2) edge (a2);
  \path (a2) edge (b2);
  \path (a3) edge (b1) edge (b3) edge[bend left=30] (a5);
  \path (a4) edge (b2) edge (b5) edge (a5);
  \path (a5) edge (a6);
  \path (a6) edge (b5) edge (b6) edge (b7) edge (a7) edge[bend left=30] (a8);
  \path (a7) edge (b6) edge (b7);
  \path (a8) edge (b8) edge (a9);
  \path (a9) edge (b8);
  \tikzstyle{every node}=[above,inner sep=3pt]
  \draw (1.26,1.16)--(1.26,1.32)--node {$B_\sigma(u_1)$}(2.74,1.32)--(2.74,1.16);
  \draw (3.26,1.16)--(3.26,1.32)--node {$B_\sigma(u_3)$}(3.74,1.32)--(3.74,1.16);
  \draw (4.76,1.16)--(4.76,1.32)--node {$B_\sigma(u_4)$}(5.24,1.32)--(5.24,1.16);
  \draw (5.76,1.16)--(5.76,1.32)--node {$B_\sigma(u_6)$}(7.24,1.32)--(7.24,1.16);
  \draw (8.26,1.16)--(8.26,1.32)--node {$B_\sigma(u_8)$}(8.74,1.32)--(8.74,1.16);
\end{tikzpicture}
\caption{A legal ordering $\sigma=(u_1,\ldots,u_9)$ of $V\setminus I$}
\label{fig:legal}
\end{figure}

Recall that if $\pi$ is a linear ordering of a set $X$ and $x\in X$, then we define
\begin{equation*}
\pi^-(x)=\bigl\{y\in X\colon y<_\pi x\bigr\}.
\end{equation*}
Fix an independent set $I\subset V$.
A linear ordering $\sigma$ of $V\setminus I$ is \emph{legal} if the following holds for every vertex $v\in V\setminus I$ other than the first one in the order~$\sigma$:
\begin{equation*}
v\in N(\sigma^-(v))\cup N\bigl(N(\sigma^-(v))\cap I\bigr).
\end{equation*}
That is, legal orderings are the orderings in which the vertices in $V\setminus I$ can be taken during the game.
For every legal linear ordering $\sigma$ of $V\setminus I$ and every $v\in V\setminus I$, define
\begin{equation*}
B_\sigma(v)=\bigl(N(v)\cap I\bigr)\setminus\bigl(N(\sigma^-(v))\cap I\bigr).
\end{equation*}
That is, $B_\sigma(v)$ is the set of vertices in $I$ that become available after $v$ has been taken, assuming that the vertices of $V\setminus I$ are being taken in the order $\sigma$.
See Figure \ref{fig:legal}.
For fixed $\sigma$, the non-empty sets $B_\sigma(v)$ form a partition of $I$.
For every legal linear ordering $\sigma$ of $V\setminus I$ and every $v\in V\setminus I$, if $B_\sigma(v)\neq\emptyset$, then choose a vertex $u_\sigma(v)$ in $B_\sigma(v)$ with maximum weight.
Since $B_\sigma(v)$ may contain several vertices of maximum weight, to achieve consistency in the choice of $u_\sigma(v)$, always select $u_\sigma(v)$ as the earliest vertex in some globally prescribed linear ordering on $I$ among the vertices from $B_\sigma(v)$ of maximum weight.
This way, for any two legal linear orderings $\sigma$ and $\sigma'$ of $V\setminus I$, if $B_\sigma(v)\supset B_{\sigma'}(v)\neq\emptyset$ and $u_\sigma(v)\neq u_{\sigma'}(v)$, then $u_\sigma(v)\notin B_{\sigma'}(v)$.
For fixed $\sigma$, the vertices $u_\sigma(v)$ are distinct.
Define
\begin{equation*}
U_\sigma=\bigl\{u_\sigma(v)\colon v\in V\setminus I\text{ and }B_\sigma(v)\neq\emptyset\bigr\}.
\end{equation*}
Note that $U_\sigma\subset I$.

\begin{lemma}
\label{lem:strat-legal}
For every independent set\/ $I\subset V$ and every legal linear ordering\/ $\sigma$ of\/ $V\setminus I$, Alice has a strategy in the graph sharing game on\/ $G$ to collect vertices of total weight at least\/ $\tfrac{1}{2}w(I\setminus U_\sigma)$.
\end{lemma}

\begin{proof}
Let $\sigma$ be a legal linear ordering of $V\setminus I$.
We can assume without loss of generality that all vertices in $V\setminus I$ have zero weight.
The strategy of Alice goes as follows.
Start by taking from $V\setminus I$ the first vertex in the order $\sigma$.
In every subsequent move, if a vertex in $I$ is available, then take one with maximum weight.
Otherwise, take from $V\setminus I$ the first non-taken vertex in the order $\sigma$.
Such a vertex is always available, as $\sigma$ is legal.

It suffices to show that Bob's final outcome is at most Alice's outcome plus $w(U_\sigma)$.
To this end, for every vertex $u\in I$ collected by Bob, we bound $w(u)$ from above by $w(v)$ or $w(u_\sigma(v))$, where $v$ is the vertex taken by Alice in her directly preceding move.
Consider the position in the game just before Alice takes $v$.
If $v\in I$, then both $u$ and $v$ are available at this position and thus $w(u)\leq w(v)$.
Otherwise, Alice's move taking $v$ makes $u$ available, and hence $u\in B_\sigma(v)$, which implies $w(u)\leq w(u_\sigma(v))$.
\end{proof}

\subsection{Proof of Theorem \ref{thm:intro-game}}

\begin{lemma}
\label{lem:final}
For every\/ $n\in\setN^+$, there is\/ $c_n\in(0,1]$ such that if\/ $G$ contains no subdivision of\/ $K_n$, then at least one of the following conditions holds:
\begin{enumerate}
\item\label{item:final-1} There is a sparse set\/ $S\subset V$ such that\/ $w(S)\geq c_nw(G)$.
\item\label{item:final-2} There are an independent set\/ $I\subset V$ and a legal linear ordering\/ $\sigma$ of\/ $V\setminus I$ such that\/ $w(I\setminus U_\sigma)\geq c_nw(G)$.
\end{enumerate}
\end{lemma}

\begin{proof}
Fix $n\in\setN^+$.
Let $p_n\in\setN$ be the constant claimed by Theorem \ref{thm:arrangeable}.
We show that it is enough to set
\begin{equation}
\label{eq:final-c}
c_n=\frac{1}{(p_n^2+4p_n+5)(p_n^2+4p_n+3)}.
\end{equation}

The proof proceeds as follows.
We fix a heavy independent set $I$ and a legal linear ordering $\sigma$ of $V\setminus I$ appropriately.
Recall that $U_\sigma\subset I$.
If $w(U_\sigma)$ is small, then we show that \ref{item:final-2} holds, otherwise we find a heavy sparse set $S\subset U_\sigma$ to satisfy \ref{item:final-1}.

Let $G$ be a weighted connected graph with vertex set $V$ and with no subdivision of $K_n$.
By Theorem \ref{thm:arrangeable}, $G$ is $p_n$-arrangeable.
Therefore, in view of Observation \ref{obs:arrangeable}, there is a linear ordering $\pi$ of $V$ with the following properties:
\begin{alignat}{2}
\label{eq:arrangeable1}
\size{N_\pi^-(v)}&\leq p_n+1&\quad&\text{for }v\in V,\\
\label{eq:arrangeable2}
\size{N_\pi^-(N[v])\cap\pi^-(v)}&\leq p_n^2+4p_n+2&\quad&\text{for }v\in V.
\end{alignat}
By \eqref{eq:arrangeable2}, we can color $V$ greedily according to the ordering $\pi$ so as to distinguish any two vertices $u$ and $v$ such that $u\in N_\pi^-(N[v])\cap\pi^-(v)$, and we use at most $p_n^2+4p_n+3$ colors to this end.
Let $I$ be a set of vertices with the same color which maximizes $w(I)$.
It follows that $I$ is independent and the following holds:
\begin{gather}
\label{eq:final-I}
\text{if $\size{N(v)\cap I}\geq 2$, then $N(v)\cap I\subset\pi^+(v)$, for $v\in V$},\\
\label{eq:final-w(I)}
w(I)\geq\frac{w(G)}{p_n^2+4p_n+3}.
\end{gather}

Let $\sigma$ be a legal linear ordering of $V\setminus I$ minimizing $w(U_\sigma)$.
If
\begin{equation*}
w(I\setminus U_\sigma)\geq(p_n^2+4p_n+3)c_nw(I),
\end{equation*}
then, by \eqref{eq:final-w(I)}, we have
\begin{equation*}
w(I\setminus U_\sigma)\geq c_nw(G),
\end{equation*}
so \ref{item:final-2} holds for $I$ and $\sigma$.
Thus assume
\begin{equation}
\label{eq:w(I-U_sigma)}
w(I\setminus U_\sigma)<(p_n^2+4p_n+3)c_nw(I).
\end{equation}
Define
\begin{alignat*}{1}
V^\star&=\bigl\{v\in V\setminus I\colon\size{N(v)\cap I}\geq 2\bigr\},\\
V_\sigma^\star&=\bigl\{v\in V\setminus I\colon\size{N(v)\cap U_\sigma}\geq 2\bigr\}\subset V^\star.
\end{alignat*}
By \eqref{eq:final-I}, we have $N(v)\cap I\subset\pi^+(v)$ for every $v\in V^\star$.
This and \eqref{eq:arrangeable1} yield
\begin{equation}
\label{eq:final-back}
\size{N(u)\cap V^\star}\leq p_n+1\quad\text{for }u\in I.
\end{equation}

The next part of the proof is to find a set $U_\sigma^\star\subset U_\sigma$ such that $\size{N(v)\cap U_\sigma^\star}\leq 2$ for every $v\in V_\sigma^\star$.
This will imply that every vertex in $U_\sigma^\star$ has at most $p_n+1$ other vertices in $U_\sigma^\star$ at distance $2$, and thus $U_\sigma^\star$ can be partitioned into at most $p_n+2$ sparse sets.

Let $v\in V_\sigma^\star$.
Let $q_\sigma^v$ be the first vertex in the order $\sigma$ such that $N(v)\cap B_\sigma(q_\sigma^v)\neq\emptyset$.
By the definition of $V_\sigma^\star$, we have $\size{N(v)\cap U_\sigma}\geq 2$, so $q_\sigma^v$ exists and $q_\sigma^v<_\sigma v$.
Let $\sigma^v$ be the linear ordering of $V(G)\setminus I$ obtained from $\sigma$ by moving $v$ to the first position after $q_\sigma^v$.
It is clear that $\sigma^v$ is legal.
Moreover, we have
\begin{gather}
\label{eq:new-B_sigma}
\begin{alignedat}{2}
B_{\sigma^v}(x)&=B_\sigma(x)&\quad&\text{for }x\in\sigma^-(q_\sigma^v)\cup\{q_\sigma^v\},\\
B_{\sigma^v}(x)&=B_\sigma(x)\setminus B_{\sigma^v}(v)&\quad&\text{for }x\in\sigma^+(q_\sigma^v)\setminus\{v\}.
\end{alignedat}
\end{gather}
Define
\begin{equation*}
Y_\sigma^v=U_\sigma\setminus U_{\sigma^v},\qquad Z_\sigma^v=U_{\sigma^v}\setminus U_\sigma.
\end{equation*}
That is, $Y_\sigma^v$ are the vertices removed from $U_\sigma$, and $Z_\sigma^v$ are the vertices added to $U_\sigma$.
Since $\sigma$ has been chosen so as to minimize $w(U_\sigma)$, we have
\begin{gather}
\notag
w(Y_\sigma^v)\leq w(Z_\sigma^v),\\
\label{eq:charging-bound}
\sum_{v\in V_\sigma^\star}w(Y_\sigma^v)\leq\sum_{v\in V_\sigma^\star}w(Z_\sigma^v)=\sum_{u\in I\setminus U_\sigma}\size{\{v\in V_\sigma^\star\colon u\in Z_\sigma^v\}}\cdot w(u).
\end{gather}
The definition of $\sigma^v$ yields $N(v)\cap B_{\sigma^v}(x)=\emptyset$ for $x\notin\{q_\sigma^v,v\}$.
Therefore,
\begin{equation}
\label{eq:two-neighbors}
N(v)\cap(U_\sigma\setminus Y_\sigma^v)\subset N(v)\cap U_{\sigma^v}\subset\{u_{\sigma^v}(q_\sigma^v),u_{\sigma^v}(v)\},
\end{equation}
That is, $v$ has at most two neighbors in $U_\sigma\setminus Y_\sigma^v$.
Let $u\in Z_\sigma^v=U_{\sigma^v}\setminus U_\sigma$.
By \eqref{eq:new-B_sigma}, one of the following holds:
\begin{itemize}
\item $u=u_{\sigma^v}(v)\in N(v)\cap I$; hence $v\in N(u)\cap V^\star$.
\item $u=u_{\sigma^v}(x)\in N(x)\cap I$ for some $x\in V\setminus I$ such that $u_\sigma(x)\in B_{\sigma^v}(v)\subset N(v)\cap I$; moreover, since $u_\sigma(x)$ and $u_{\sigma^v}(x)$ are two distinct vertices in $N(x)\cap I$, we have $x\in V^\star$; hence $x\in N(u)\cap V^\star$ and $v\in N(u_\sigma(x))\cap V^\star$.
\end{itemize}
Therefore,
\begin{gather}
\notag
\{v\in V_\sigma^\star\colon u\in Z_\sigma^v\}\subset(N(u)\cap V^\star)\cup\bigcup_{x\in N(u)\cap V^\star}(N(u_\sigma(x))\cap V^\star),\\
\label{eq:charging-times}
\begin{split}
\size{\{v\in V_\sigma^\star\colon u\in Z_\sigma^v\}}&\leq\size{N(u)\cap V^\star}+\sum_{x\in N(u)\cap V^\star}\size{N(u_\sigma(x))\cap V^\star}\\
&\leq(p_n+1)(p_n+2)\mskip 117mu\text{by \eqref{eq:final-back}}.
\end{split}\\
\label{eq:charging-total}
\begin{alignedat}{2}
\sum_{v\in V_\sigma^\star}w(Y_\sigma^v)&\leq\sum_{u\in I\setminus U_\sigma}\size{\{v\in V_\sigma^\star\colon u\in Z_\sigma^v\}}\cdot w(u)&\quad&\text{by \eqref{eq:charging-bound}}\\
&\leq(p_n+1)(p_n+2)w(I\setminus U_\sigma)&\quad&\text{by \eqref{eq:charging-times}}\\
&<(p_n^2+3p_n+2)(p_n^2+4p_n+3)c_nw(I)&\quad&\text{by \eqref{eq:w(I-U_sigma)}}.
\end{alignedat}
\end{gather}
Define
\begin{equation*}
U_\sigma^\star=U_\sigma\setminus\bigcup_{v\in V_\sigma^\star}Y_\sigma^v.
\end{equation*}
It follows that
\begin{gather}
\label{eq:w(U^star)}
\mskip-15mu\begin{alignedat}{2}
w(U_\sigma^\star)&\geq w(U_\sigma)-\sum_{v\in V_\sigma^\star}w(Y_\sigma^v)\\
&>\bigl(1-(p_n^2+3p_n+3)(p_n^2+4p_n+3)c_n\bigr)w(I)&\quad&\text{by \eqref{eq:w(I-U_sigma)} and \eqref{eq:charging-total}}\\
&=(p_n+2)(p_n^2+4p_n+3)c_nw(I)&\quad&\text{by \eqref{eq:final-c}}\\
&\geq(p_n+2)c_nw(G)&\quad&\text{by \eqref{eq:final-w(I)}}.
\end{alignedat}
\end{gather}

Fix a vertex $v\in U_\sigma^\star$.
If a vertex $u\in U_\sigma^\star$ is at distance $2$ from $v$ in $G$, then $u$ and $v$ share a neighbor $x_u\in V\setminus I$.
Since $\size{N(x_u)\cap I}\geq 2$, we have $x_u\in N(v)\cap V^\star$.
By \eqref{eq:two-neighbors}, we have $N(x_u)\cap U_\sigma^\star=\{u,v\}$.
Therefore, by \eqref{eq:final-back}, $v$ is at distance $2$ from at most $p_n+1$ other vertices in $U_\sigma^\star$.
It follows that the vertices in $U_\sigma^\star$ can be colored (greedily) with at most $p_n+2$ colors so that no two of them at distance $2$ in $G$ receive the same color.
Let $S$ be a color class with maximum weight.
It follows that $S$ is sparse in $G$ and, by \eqref{eq:w(U^star)},
\begin{equation*}
w(S)\geq\frac{w(U_\sigma^\star)}{p_n+2}>c_nw(G).
\end{equation*}
This shows that \ref{item:final-1} holds for $S$.
\end{proof}

\begin{proof}[Proof of Theorem \ref{thm:intro-game}]
Fix $n\in\setN^+$.
Let $c'_n\in(0,1]$ be the constant claimed by Lemma \ref{lem:final} and $c''_n\in(0,1]$ be the constant claimed by Corollary \ref{cor:strat-sparse}.
We show that it is enough to set
\begin{equation*}
c_n=c'_n\min\{c''_n,\tfrac{1}{2}\}.
\end{equation*}

Let $G$ be a weighted connected graph with an odd number of vertices and with no subdivision of $K_n$.
By Lemma \ref{lem:final}, at least one of the following~holds:
\begin{enumerate}
\item\label{item:intro-game-1} There is a sparse set\/ $S\subset V(G)$ such that\/ $w(S)\geq c'_nw(G)$.
\item\label{item:intro-game-2} There are an independent set\/ $I\subset V(G)$ and a legal linear ordering\/ $\sigma$ of\/ $V(G)\setminus I$ such that\/ $w(I\setminus U_\sigma)\geq c'_nw(G)$.
\end{enumerate}

Suppose that \ref{item:intro-game-1} holds.
Let $G'$ be the graph obtained from $G$ by resetting the weights of all vertices in $V(G)\setminus S$ to zero.
Hence
\begin{equation*}
w(G')=w_G(S)\geq c'_nw(G).
\end{equation*}
By Corollary \ref{cor:strat-sparse}, Alice has a strategy in the game on $G'$ to collect vertices of total weight at least $c'_nw(G')\geq c'_nc''_nw(G)\geq c_nw(G)$.
The same strategy gives Alice at least $c_nw(G)$ in the game on $G$.

If \ref{item:intro-game-2} holds, then, by Lemma \ref{lem:strat-legal}, Alice has a strategy in the game on $G$ to collect vertices of total weight at least $\tfrac{1}{2}w(I\setminus U_\sigma)\geq\tfrac{1}{2}c'_nw(G)\geq c_nw(G)$.
\end{proof}

\section*{Acknowledgments}

The paper is based on results of Adam Gągol's master's thesis \cite{Gag-master} and Bartosz Walczak's PhD thesis \cite{Wal-phd}.
We thank anonymous reviewers for valuable comments on the initial version of this paper.

\end{document}